\documentclass{amsart}

\usepackage{graphicx}
\usepackage{amsmath,amssymb} 
\usepackage{enumerate}
\usepackage{bbm}
\usepackage{tikz} 
\usepackage{hyperref}

\newtheorem{theorem}{Theorem}[section]
\newtheorem{lemma}[theorem]{Lemma}
\newtheorem{prop}[theorem]{Proposition}
\newtheorem{corollary}[theorem]{Corollary}
\theoremstyle{definition}
\newtheorem{definition}[theorem]{Definition}

\theoremstyle{remark} 
\newtheorem{remark}[theorem]{Remark}

\numberwithin{equation}{section}

\begin{document}

\title[fourth order cubic nonlinear Schr\"odinger equation]
{Low regularity a priori estimates for the fourth order cubic nonlinear Schr\"odinger equation}
\author{Kihoon seong }
\address{Department of Mathematical Sciences, Korea Advanced Institute of Science
	and Technology, 291 Daehak-ro, Yuseong-gu, Daejeon 34141, Korea}
\email{hun1022kr@kaist.ac.kr}  

\begin{abstract}
We consider the low regularity behavior of the fourth order cubic nonlinear Schr\"odinger equation (4NLS)	
\begin{align*}
\begin{cases}
i\partial_tu+\partial_x^4u=\pm \vert u \vert^2u, \quad(t,x)\in  \mathbb{R}\times \mathbb{R}\\
u(x,0)=u_0(x)\in H^s\left(\mathbb{R}\right).
\end{cases} 
\end{align*}
In \cite{Seong2018}, the author showed that this equation is globally well-posed in $H^s, s\geq -\frac{1}{2}$ and ill-posedness in the sense that the solution map fails to be uniformly continuous for $-\frac{15}{14}<s<-\frac{1}{2}$. Therefore, $s=-\frac{1}{2}$ is the lowest regularity that can be handled by the contraction argument. In spite of this ill-posedness result, we obtain a priori bound below $s<-1/2$. This a priori estimate guarantees the existence of a weak solution for $-3/4<s<-1/2$. But we cannot establish full well-posedness because of the lack of energy estimate of differences of solutions. Our method is inspired by Koch-Tataru \cite{KT2007}. We use the $U^p$ and $V^p$ based spaces adapted to frequency dependent time intervals on which the nonlinear evolution can be still described by linear dynamics.
\end{abstract}

\maketitle


\section{Introduction}

In this paper, we study the Cauchy problem for the fourth order cubic nonlinear Schr\"odinger equation on $\mathbb{R}$:
\begin{align}\label{eqn:fourth order NLS}\tag{4NLS}
\begin{cases}
i\partial_tu+\partial_x^4u=\mu \vert u \vert^2u, \quad(t,x)\in  \mathbb{R}\times \mathbb{R}\\
u(x,0)=u_0(x)\in H^s\left(\mathbb{R}\right), 
\end{cases} 
\end{align}
where $u$ is a complex-valued function and $\mu=\pm 1$. This equation is called defocusing when the sign of nonlinear term is negative and focusing when the sign is positive. This equation is also known as the biharmonic NLS. The $\eqref{eqn:fourth order NLS}$ has been studied in the context of stability of solitons in magnetic materials. For more physical background see \cite{Karpman1996},\cite{KS2000}.

This equation is also a Hamiltonian PDE with the following Hamiltonian: 
\begin{align}\label{eqn:Hamiltonian}
H(u(t))=\frac{1}{2}\int_{\mathbb{R}}\vert \partial_x^2 u(t) \vert^2 \,dx\pm \frac{1}{4} \int_{\mathbb{R}}\vert u(t) \vert^4\, dx.
\end{align} 
Moreover, the mass $M(u(t))$ is defined by 
\begin{align}\label{eqn:mass}
M(u(t))=\int_{\mathbb{R}}\vert u(t) \vert^2\,dx.
\end{align}
This Hamiltonian $\eqref{eqn:Hamiltonian}$ and mass $\eqref{eqn:mass}$ are conserved under the $\eqref{eqn:fourth order NLS}$ flow.

 The $\eqref{eqn:fourth order NLS}$ is invariant with respect to the scaling 
\begin{align}\label{eqn: scaling symmetry}
u(t,x) \to \lambda^2u(\lambda^4t,\lambda x)
\end{align}
Therefore, the scale invariant homogeneous space is $\dot{H}^{-\frac{3}{2}}$.
 In general, we have
\begin{align}
\Vert u_{0,\lambda} \Vert_{\dot{H}^s\left(\mathbb{R}\right) }=\lambda^{s+\frac{3}{2}}\Vert u_0 \Vert_{\dot{H}^s\left(\mathbb{R}\right) }.
\end{align}
The $\eqref{eqn:fourth order NLS}$ is globally well-posed for initial data $u_0 \in L^2$.
Therefore, it is natural to ask whether the well-posedness also holds in negative Sobolev spaces between $H^{-\frac{3}{2}} $ and $L^2$.  

Let us investigate the one-dimensional cubic NLS:
\begin{align}\label{eqn: cubic NLS,focusing defocusing}\tag{NLS}
\begin{cases}
i\partial_tu+\partial_x^2u=\pm \vert u \vert^2u, \quad(t,x)\in  \mathbb{R}\times \mathbb{R}\\
u(x,0)=u_0(x)\in H^s\left(\mathbb{R}\right).
\end{cases} 
\end{align}
We look at Galilean invariance: if $u$ is a solution of $\eqref{eqn: cubic NLS,focusing defocusing}$ with initial data $u_0$, then
\begin{align} 
u_N(t,x)=e^{ixN}e^{-itN^2}u(t,x-2Nt) 
\end{align}   
is a solution to the same equation $\eqref{eqn: cubic NLS,focusing defocusing}$ with initial data $e^{ixN}u_0(x)$. 
As a consequence of the Galilean invariance, the flow map cannot be uniformly continuous in $H^s, s<0$. The details are presented in \cite{CCT2003},\cite{KPV2001}.
As for $\eqref{eqn:fourth order NLS}$, there is no Galilean symmetry and hence we can pursue the well-posedness theory in negative regularity regime $s<0$.

In \cite{Seong2018}, the author showed that $s=-\frac{1}{2}$ is the sharp regularity threshold for which the well-posedness can be handled by Picard iteration argument. More precisely, for $s\geq -1/2$, the author proved that the $\eqref{eqn:fourth order NLS}$ is to be locally and globally well-posed in $H^s$ and below $s<-1/2$, it is ill-posed in the sense that the flow map fails to be uniformly continuous in $H^s$.

Although the flow map is not uniformly continuous for $s<-1/2$, we may have well-posedness with only continuous dependence on the initial data. Therefore, our final goal is to fill the gap between $H^{-3/2}$ and $H^{-1/2}$.
In fact, to prove the well-posedness, we need to show a priori $H^s$ bounds for the solutions and also establish continuous dependence on the initial data. In this paper, we prove a priori estimates up to $s>-3/4$. As a corollay, we show the existence of a weak solution for any initial data $u_0\in H^s, s>-3/4$. Our method is inspired by Koch-Tataru \cite{KT2007} and Christ-Colliander-Tao \cite{CCT2008}.

The main results of this paper are the following a priori estimate and the existence of weak solution. 
\begin{theorem}[A priori estimate]\label{thm:main theorem}
Let $-\frac{3}{4}<s<-\frac{1}{2}$. Then for any $M> 0$, there exists time $T$ and constant $C$ so that for initial data $u_0 \in \mathcal{S}$ satisfying
\begin{align*}
\Vert u_0 \Vert_{H^s}\leq M,
\end{align*}

the unique solution $u \in C\left([0,T]; \mathcal{S} \right)$ to $\eqref{eqn:fourth order NLS}$(focusing or defocusing) satisfies 
\begin{align}\label{eqn:uniform bound}
\sup\limits_{t\in [0,T]}\Vert u(t) \Vert_{H_x^s}\leq C \Vert u_0 \Vert_{H_x^s}.
\end{align}
\end{theorem}
Using the uniform bound $\eqref{eqn:uniform bound}$ together with the uniform bound on nonlinearity
\begin{align*}
\Vert \chi_{[0,T]} u \Vert_{X^s}+\Vert \chi_{[0,T]}  \vert u \vert^2 u \Vert_{Y^s}\lesssim \Vert u_0\Vert_{H^s},
\end{align*}
which is a byproduct of our analysis in proving Theorem \ref{thm:main theorem}, one may also prove the existence of weak solution by following the similar argument as in \cite{CCT2008}. The spaces $X^s,Y^s$ are defined in Section \ref{sec:function space}.

\begin{corollary}[Existence of weak solution.]\label{cor:weak solution}
Let $-\frac{3}{4}<s<-\frac{1}{2}$. For any $M>0$ there exist time $T$ and constant $C$ so that for any initial data in $H^s$ satisfying
\begin{align*}
\Vert u_0 \Vert_{H^s} \leq M,
\end{align*}
there exists a weak solution $u\in C\left(\left[0,T\right]; H^s \right)\cap X^s$ to $\eqref{eqn:fourth order NLS}$ which solves the equation in the sense of distributions and satisfies
\begin{align*}
\Vert u \Vert_{L_t^\infty H^s}+\Vert \chi_{[0,T]} u \Vert_{X^s}+\Vert \chi_{[0,T]} \vert u \vert^2 u \Vert_{Y^s} \leq C \Vert u _0 \Vert_{H^s}.
\end{align*}
\end{corollary} 
The solution obtained by Corollary \ref{cor:weak solution} is a weak limit of smooth solutions with smooth initial data approximating given data. We call these solutions weak solutions because we do not have any uniqueness or continuous dependence in $H^s, -3/4<s<-1/2$. 
\begin{remark}\label{rem:small data remark}
We can always rescale the initial data and hence it suffices to prove the theorem in small data case $M\ll 1$.	
\end{remark}

In \cite{Seong2018}, by just taking advantage of dispersive smoothing effects (bilinear Strichartz estimates, nonresonant interactions), the author proved the local and global well-posedness for $s \geq -1/2$.
The main part of the local well-posedness is to show that the following trilinear estimate 
\begin{align*}
\| u_1 \overline{u_2}u_3\|_{X^{s,b-1}}\leq \|u_1 \|_{X^{s,\frac{1}{2}+}}\|u_2 \|_{X^{s,\frac{1}{2}+}}\|u_3 \|_{X^{s,\frac{1}{2}+}}	
\end{align*}
holds for $s\geq -1/2$. However, in \cite{Seong2018}, the author display an example that for $s<-1/2$, the above trilinear estimate fails because of the strong resonant interaction of high-high-high to high.

To deal with these resonant interaction, in this paper we use the short time strucutre. More precisely, we use the functions spaces adapted to a short time interval depending on the dyadic size of spatial frequencies. Then these high-high-high to high resonant interaction is overcomed so that we can prove the trilinear estimate below $s<-1/2$.

In Remark \ref{remark:short time length}, one can see that nonlinear solution which is localized at frequency $N$ can be still illustrated by linear dynamics up to the time scale $N^{4s+2}$. In fact, for $s\geq -1/2$, we already have a local well-posedness result. Therefore, in view of perturbation, nonlinear solituions behave like a linear solution over a time interval which is independent of frequency $N$. For $s\geq -\frac{1}{2}$, one can observe  $N^{4s+2}\gtrsim 1$ for all large $N \gg 1$. Therefore, there is an uniform time interval in $N$ so that nonlinear solutions localized at each frequency $N$ show linear dynacims 
on that time interval. But for $s<-1/2$, in order for a nonlinear solution to follow linear dynamics, the time scale must depend on the size of spatial frequencies. Observe that for $s<-\frac{1}{2}$, $N^{4s+2} \ll 1$ for all large $N\gg 1$. In contrast to the case $s\geq -\frac{1}{2}$, it means that different time scales are required for nonlinear solution localized at frequency $N$ to follow linear dynamics.

We also use the $U^p$ and $V^p$ spaces. These spaces have been originally introduced in unpublished work of Tataru on wave maps. In Koch-Tataru \cite{KT2007},  they also used $U^p$ and $V^p$ spaces adapted to time intervals depending on the size of spatial frequencies.
In Section \ref{sec:function space}, we define the function spaces employed in our analysis.

We briefly review the difference between Picard iteration method and the short time structure method. In fact, the latter is less perturbative than the former. We consider the following evolution equation: $\partial_tu -Lu=\mathcal{N}\left(u\right)$, where $Lu$ is a linear part and $\mathcal{N}(u)$ is a homogeneous nonlinearity of degree $p$. 
Then the usual Picard iteration method needs to establish the following two estimates:
\begin{align*}
\text{Linear:}& \quad\quad &\Vert u \Vert_{F^s}\lesssim& \Vert u_0\Vert_{H^s}+\Vert \mathcal{N}(u) \Vert_{N^s},\\
\text{Nonlinear:}& \quad\quad &\Vert \mathcal{N}(u)\Vert_{N^s}\lesssim& \Vert u \Vert_{F^s}^p,
\end{align*}  
where $F^s$ is the space to measure the solutions and $N^s$ is the space to measure the nonlinearity. After obtaining these two estimates, we can apply the fixed point argument to obtain the local well-posedness in $H^s$. For the short time strucutre method, by using the gain coming from the short time scale, nonlinear estimate can be improved up to lower regularity levels compared with the previous method. However, linear estimates are worse than before. To address these expense, we need to establish the additional energy estiamtes. In summary, we must establish the following three estimates:
\begin{align*}
\text{Linear:}&\quad\quad &\Vert u \Vert_{X^s}\lesssim& \Vert u \Vert_{\ell^2L_t^\infty H_x^s}+\Vert \mathcal{N}(u) \Vert_{Y^s},\\
\text{Nonlinear:}&\quad\quad &\Vert \mathcal{N}(u) \Vert_{Y^s}\lesssim& \Vert u \Vert_{X^s}^p,\\
\text{Energy:}&\quad\quad &\Vert u \Vert_{\ell^2L_t^\infty H_x^s}\lesssim& \Vert u_0\Vert_{H^s}+\Vert u \Vert_{X^s}^p,
\end{align*}
where $X^s, Y^s $ and energy space $\ell^2L_t^\infty H_x^s $ are presented in Section \ref{sec:function space}. Then by using a continuity argument combining with above three estimates, one can establish an a priori bound and hence can prove the existence of solutions by a compactness argument. In order to obtain the energy bound, we need to use a normal form technique. In applying the normal form reduction, we need to take the expense of introducing higher order multilinear terms.

Therefore, the process of obtaining a priori bound is divided into two main steps. One is to prove the following trilinear estimates $\eqref{eqn:trilinear estimate}$: Let $-3/4<s<-1/2$. Then we have
\begin{align*}
\Vert u_1u_2u_3 \Vert_{Y^s}\lesssim \Vert u_1 \Vert_{X^s} \Vert u_2 \Vert_{X^s} \Vert u_3 \Vert_{X^s},
\end{align*}
where $Y^s $ is a function space to measure the nonlinear term in $\eqref{eqn:fourth order NLS}$ and $X^s$ is a function space to measure the solutions for $\eqref{eqn:fourth order NLS}$. These function spaces are defined in Section \ref{sec:function space}.

In the nonlinear interactons which result in high frequency $N$, we can use the gain $\vert J \vert=N^{4s+2}$ occuring from the short time structure and hence, combining the dispersive smoothing effects(e.g. Strichartz estimates $\eqref{eqn:Stricharz L-P piece N}$ and bilinear Strichartz estiamtes $\eqref{eqn:bilinear not log}, \eqref{eqn:log interpolation}$), we can obtain the triliner estimates for all $s<-1/2$. 
However, there is a trade-off of using the short time strucutre. One can expect a loss resulted from summation of short time intervals. More precisely, in the nonlinear interactions which result in low frequency $N$, there is a loss of derivative originated from the interval summation. 
In fact, in proving trilinear estimates $\eqref{eqn:trilinear estimate}$, high-high-high to low interaction is the worst case in terms of interval summation losses. To address this side effect of short time strucutre, we need to use the another dispersive smoothing effects. We can observe that the high-high-high to low interaction is a nonresonant interaction. More precisely, either the output or at least one of the inputs must have high modulation:
under $\xi=\xi_1-\xi_2+\xi_3, \tau=\tau_1-\tau_2+\tau_3$, we have
\begin{align*}
&\vert \tau_1-\xi_1^4 \vert + \vert  \tau_2-\xi_2^4 \vert +\vert \tau_3-\xi_3^4 \vert +\vert \tau-\xi^4 \vert \\ \gtrsim&\left \vert \left(\xi-\xi_1\right)\left(\xi-\xi_3\right)\left(\xi_1^2+\xi_2^2+\xi_3^2+\xi^2+2\left(\xi_1+\xi_3\right)^2 \right) \right\vert \\
\gtrsim& M_1 M_3^3, 
\end{align*}  
where the size of frequencies $\left\{\xi, \xi_1,\xi_2,\xi_3  \right\}$ is $\left\{M,M_1,M_2,M_3  \right\}$ with $ \vert \xi \vert\approx M$, $M\ll M_1 \lesssim M_2\approx M_3$. Furthermore, we use the local smoothing effect. Therefore by exploiting these high modulation gain in the high-high-high to low interaction and local smoothing effect, we are able to weaken the interval summation losses and hence can prove trilinear estimate $\eqref{eqn:trilinear estimate}$ up to $s>-\frac{3}{4}$. The details are presented in Lemma \ref{lem:improved trilinear}. In fact, by using only the bilinear smoothing effect $\eqref{eqn:bilinear not log},\eqref{eqn:log interpolation}$ without using the above high modulation gain and local smoothing effect, we can prove tirlinear estimate $\eqref{eqn:trilinear estimate}$ up to $s>-\frac{5}{7}$. For more details, see Remark \ref{rem:-5/7 trilinear}.

In Lemma \ref{lem:nonlinear estimate}, we prove the following linear estimates:    
Let $u$ be a solution of $i\partial_t u -\partial_x^4u=f$. Then we have 
\begin{align}\label{eqn:right handside 2}
\Vert u \Vert_{X^s}\lesssim \Vert u \Vert_{\ell^2 L^{\infty}H^s }+\Vert f \Vert_{Y^s}.
\end{align}    
Therefore, we use  trilinear estimate $\eqref{eqn:trilinear estimate}$ to control the second term on the right hand side of $\eqref{eqn:right handside 2}$.
Therefore, the other part is to prove the following energy estimates $\eqref{eqn:energy estimates N.}$ to control the first term on the right hand side of $\eqref{eqn:right handside 2}$: Let $-\frac{3}{4}< s<-\frac{1}{2}$ and $u$ be a solution to $\eqref{eqn:fourth order NLS}$. On the time interval $[0,1]$, we have the following energy estimates $\eqref{eqn:energy estimates N.}$
\begin{align*}
\Vert u \Vert_{\ell_N^2 L_t^\infty H^s} \lesssim \Vert u_0 \Vert_{H^s}+\Vert u \Vert_{X^s}^3.
\end{align*}
 In Section \ref{sec:energy estimate}, energy estimate is established by using a high frequency damped multiplier. This method is a modification of the $I$-method introduced by Colliander-Keel-Staffilani-Takaoka-Tao \cite{CKSTT2001}, \cite{CKSTT2002}, \cite{CKSTT2003}. In the process of obtaining energy estimates, we use the normal form technique with the function spaces relying on frequency dependent time scales. As in proving trilinear estimate $\eqref{eqn:trilinear estimate}$, there is a loss of derivative resulted from the interval summation. To deal with this loss, 
 we also use Lemma \ref{lem:improved trilinear} again. Our method can not construct energy bound $\eqref{eqn:energy estimates N.}$ for differences of solutions, which is the reason we cannot establish full well-posedness in $H^s, -3/4<s<-1/2$.   
 The remaining part is to just use standard bootstrapping argument with trilinear estimates $\eqref{eqn:trilinear estimate}$ and energy estiamtes $\eqref{eqn:energy estimates N.}$. The details are presented in Section $\ref{sec:main theorem}$.

 \textbf{Organization of paper.} 
This paper is organized as follows. In Section \ref{sec:function space}, we introduce $U^p$ and $V^p$ spaces adapted to short time intervals. In Section \ref{sec:bilinear estimate}, we collect the linear and bilinear dispersive estimates used to prove the trilinear estimate and the energy estimate. These include Strichartz estimates, bilinear Strichartz estimates, local smoothing estimates and maximal function estimates. In Section \ref{sec:trilinear estimate}, the trilinear estimate is proved. To weaken the interval summation losses, we take advantage of Lemma \ref{lem:improved trilinear}.  In Section \ref{sec:energy estimate}, the energy estimate with a higher order correction term is established by using a variation of the I-method. Finally, in Section \ref{sec:main theorem}, all materials are collected to give a proof of Theorem \ref{thm:main theorem}.

\begin{remark}\label{remark:short time length}
We want to explain that a solution whose frequency is localized to $N$ behaves like a linear solution during at least $N^{4s + 2}$ time scales.

Now we assume $u$ is a solution to $\eqref{eqn:fourth order NLS}$, which is localized at frequency $N\gg 1$. We also suppose $u\approx e^{it\partial_x^4}u_0$ on $[0,T]$ for small time $T \ll 1$ with $\Vert u_0 \Vert_{H_x^s}\approx1$.
By using the Duhamel's formula, we write
\begin{align*}
u(t)=e^{it\partial_x^4}u_0\pm \int_0^t e^{i(t-t')\partial_x^4}\vert u \vert^2u(t')\,dt'.
\end{align*}
In order for our solution $u$ to follow linear dynamics on $[0,T]$, we should have
\begin{align}\label{eqn:linear dynamics}
\left \| \int_0^t e^{i(t-t')\partial_x^4}\vert u \vert^2u(t')\,dt' \right\|_{L_t^{\infty}H_x^s\left([0,T]\times \mathbb{R}\right)}\approx 1.
\end{align}
Therefore, we estimate above nonlinear term as
\begin{align*}
&\left \| \int_0^t e^{i(t-t')\partial_x^4}\vert u \vert^2u(t')\,dt' \right\|_{L_t^{\infty}H_x^s\left([0,T]\times \mathbb{R}\right)}\\
\lesssim& N^s \Vert u \overline{u} u \Vert_{L_t^1L_x^2\left([0,T]\times\mathbb{R}\right)}\\
\lesssim & T^{\frac{1}{2}} N^s \Vert u \Vert_{L_t^6L_x^6\left([0,T]\times \mathbb{R} \right)}^3
\end{align*} 
Since $u$ follows linear dynamics on $[0,T]$ i.e. $u\approx e^{it\partial_x^4}u_0  $ on $[0,T]$, $u$ satisfies the Strichartz estimate $\eqref{eqn:D^{1/3} Strichartz}$:
\begin{align*}
\Vert D^{\frac{1}{3}} e^{it\partial_x^4} u_0 \Vert_{L_t^6L_x^6\left(\mathbb{R}\times \mathbb{R}\right)}\lesssim\Vert u_0 \Vert_{L_x^2\left(\mathbb{R}\right)}.
\end{align*}
Therefore, by applying the Strichartz estimate $\eqref{eqn:D^{1/3} Strichartz}$, we have
\begin{align*}
\Vert u\Vert_{L_t^6L_x^6}\approx& \Vert e^{it\partial_x^4}u_0 \Vert_{L_t^6L_x^6\left([0,T]\times \mathbb{R}\right)}\\
\lesssim& N^{-\frac{1}{3}}\Vert u_0 \Vert_{L_x^2}\\
\approx& N^{-\frac{1}{3}-s}\Vert u_0  \Vert_{H^s}\\
\approx & N^{-\frac{1}{3}-s}   
\end{align*}
To obtain $\eqref{eqn:linear dynamics}$, we need
\begin{align*}
T^{\frac{1}{2}} N^s \Vert u \Vert_{L_t^6L_x^6\left([0,T]\times \mathbb{R} \right)}^3\lesssim T^{\frac{1}{2}} N^s \left(N^{-\frac{1}{3}-s} \right)^3 \approx 1.
\end{align*}
Therefore, we choose time scale $T\approx N^{4s+2}\ll 1$.

By observing this heuristic calculation, we will construct our function spaces to be adapted to time intervals whose length depends on the  time scale $T=N^{4s+2}$.   
\end{remark}

\textbf{Notation.} We use $A\lesssim B $ if $A\leq CB$ for some $C>0$. We use $A\approx B$ when $A\lesssim B$ and $B\lesssim A$. Moreover, we use $A\ll B$ if $A\leq \frac{1}{C}B$, where $C$ is a sufficiently large constant. We also write $A^{\pm}$ to mean $A^{\pm \epsilon}$ for any $\epsilon>0$.

Given $ p\geq 1 $, we let $p'$ be the H\"older conjugate of $p$ such that $ \frac{1}{p}+\frac{1}{p'}=1$. We denote $L^p=L^p\left(\mathbb{R}^d\right)$ be the usual Lebesgue space. We also define the Lebesgue space $L^q\left(I,L^r\right)$ be the space of measurable functions from an interval $I\subset \mathbb{R}$ to $L^r$ whose $L^q\left(I,L^r\right)$ norm is finite, where
\begin{align*}
\Vert u\Vert_{L^q\left(I,L^r \right)}=\left(    \int_I \Vert u(t) \Vert_{L^r}^q       \right)^{\frac{1}{q}}.
\end{align*}
We may write $L_t^qL_x^r\left(I\times \mathbb{R}\right)$ instead of  $L^q\left(I,L^r\right)$.

We denote the space time Fourier transform of $u(t,x)$ by $\widehat{u}(\tau,\xi)$ or $\mathcal{F}u$
\begin{align*}
\widehat{u}(\tau,\xi)=\mathcal{F}u(\tau,\xi)=\int e^{-it\tau-ix\xi}u(t,x)\,dtdx.
\end{align*}
On the other hand, the space Fourier transform  of $u(t,x)$ is denoted by
\begin{align*}
\widehat{u}(t,\xi)=\mathcal{F}_xu(t,\xi)=\int e^{-ix\xi}u(t,x)\,dx.
\end{align*}
The fractional differential operator is given via Fourier transform by
\begin{align*}
\widehat{D^{\alpha}u}(\xi)=\vert \xi \vert^{\alpha}\widehat{u}(\xi), \quad \alpha \in\mathbb{R},
\end{align*}
and the biharmonic Schr\"odinger semigroup is defined by
\begin{align}
e^{it\partial_x^4}=\mathcal{F}_x^{-1}e^{it\vert \xi \vert^4}\mathcal{F}_xg
\end{align}
for any tempered distribution $g$.
Let $\varphi:\mathbb{R} \to [0,1]$ be an even, smooth cutoff function supported on $[-2,2]$ such that $\varphi=1$ on $[-1,1]$. Given a dyadic number $N\geq 1$, we set $\varphi_1\left(\xi\right)=\varphi\left(\vert \xi \vert  \right)$ and 
\begin{align*}
\varphi_N\left(\xi\right)=\varphi\left(\frac{\vert \xi \vert}{N}\right)-\varphi\left(\frac{2\vert \xi \vert}{N} \right) 
\end{align*}
for $N\geq 2$. Then we define the Littlewood-Paley projection operator $P_N$ as the Fourier multiplier operator with symbol $\varphi_N$. Moreover, we define $P_{\leq N}$ and $P_{\geq N}$ by $P_{\leq N}=\sum\limits_{1\leq M \leq N}P_M$ and $P_{\geq N}=\sum\limits_{M\geq N}P_M$.
They commute with the derivative operator $D^\alpha$ and the semigroup $e^{it\partial_x^4}$. We also use the notation $u_N=P_Nu$ if there is no confusion.

\textbf{Acknowledgements.} The author would like to appreciate his advisor Soonsik Kwon for helpful discussion and encouragement. The author is also grateful to Chulkwang Kwak for his helpful discussion to understand well the short time structure. The author is partially supported by NRF-2018R1D1A1A09083345 (Korea).

\section{Function spaces}\label{sec:function space}\label{sec:function space} 
In this section, we set up the function spaces employed in our analysis.
We also go over the properties of function spaces $U^p$ and $V^p$ established by Koch, Tataru. These spaces have been used in developing well-posedness of dispersive equations at scaling critical regularities. The details are presented in Hadac-Herr-Koch \cite{HHK2009}, Herr-Tataru-Tzvetkov \cite{HTT1}, Koch-Tataru \cite{KT2018} and Koch-Tataru-Visan \cite{KTV2014}. 

 We take a time interval $I=[a,b), -\infty\leq a < b \leq \infty$. Let $\mathcal{Z}$ be the set of partitions $a=t_0<t_1<\cdots < t_K=b$ of $I$. We also consider functions taking values in $ L^2=L^2\left(\mathbb{R}\right)$.
\begin{definition}
Let $1\leq p <\infty$. For $\left\{t_k  \right\}_{k=0}^K \in \mathcal{Z} $ and $\left\{\phi_k  \right\}_{k=0}^{K-1}\subset L^2\left(\mathbb{R}\right)$ with $\phi_0=0$ and $\sum\limits_{k=0}^{K-1} \Vert \phi_k \Vert_{L^2}^p=1$.
We call the function $a:I \to L^2$ given by
\begin{align*}
a(t)=\sum\limits_{k=1}^K\chi_{[t_{k-1},t_k ) }(t)\phi_{k-1}
\end{align*}
a $U^p\left(I;L^2\right)$- atom.
We define the $U^p\left(I;L^2\right)$ space: 
\begin{align*}
U^p\left(I;L^2\right):=\left\{u=\sum\limits_{j=1}^{\infty} \lambda_j a_j : \text{$a_j$ is $U^p\left(I;L^2\right)$-atom, $\lambda_j \in \mathbb{C}$ such that} \sum\limits_{j=1}^{\infty}\vert \lambda_j \vert<\infty    \right\}
\end{align*}	
with norm
\begin{align*}
\Vert u\Vert_{U^p\left(I;L^2\left(\mathbb{R}\right)\right)}:=\inf \left\{ \sum\limits_{j=1}^{\infty} \vert \lambda_j \vert : u=\sum\limits_{j=1}^{\infty}\lambda_j a_j ,\; \lambda_j \in \mathbb{C},\;  a_j \;U^p\left(I;L^2\right)-\text{atom} 
  \right\}.
\end{align*}
\end{definition} 

\begin{definition}
Let $1\leq p <\infty $. We define the space $ V^p\left(I;L^2\right) $ as the space of functions on $I$ such that 
\begin{align*}
v(a)=\lim\limits_{t\rightarrow a }v(t) \; \text{exists} \quad \text{and} \quad v(b):=\lim\limits_{t\rightarrow b}v(t)=0,
\end{align*} 
and for such functions $v(t)$ we define the norm
\begin{align*}
\Vert v \Vert_{V^p\left(I;L^2\right)}=\sup\limits_{\left\{t_k \right\} \in \mathcal{Z} } \left(\sum\limits_{k}\Vert v(t_k)-v(t_{k-1}) \Vert_{L^2}^p   \right)^{\frac{1}{p}}.
\end{align*}
\end{definition}
We also use the notation $U^p=U^p\left(I;L^2\right)$ and $V^p=V^p\left(I;L^2\right)$ if there is no confusion.

\begin{lemma}[\cite{HHK2009},\cite{KT2007}]\label{lem:U^p ,V^p embedding lemma}
Fix an interval $I=[a,b)$.
\begin{enumerate}[1.]
	
\item Let $1\leq p< q <\infty$. Then we have continuous embeddings  $U^p \hookrightarrow U^q$ and $V^p \hookrightarrow V^q$ i.e. 
\begin{align}
\Vert u \Vert_{U^q}\lesssim \Vert u \Vert_{U^p} \quad \text{and} \quad \Vert u \Vert_{V^q}\lesssim \Vert u \Vert_{V^p}.
\end{align}

\item If $1\leq p < \infty$ and $u(b)=0$, then we have $U^p \hookrightarrow V^p$ i.e.
\begin{align*}
\Vert u \Vert_{V^p}\lesssim \Vert u \Vert_{U^p}.
\end{align*}

\item If $1\leq p < q \leq \infty, u(a)=0 $, and $u\in V^p$ is right continuous, then we have 
\begin{align*}
\Vert u \Vert_{U^q}\lesssim \Vert u \Vert_{V^p}.
\end{align*}

\end{enumerate}
\end{lemma}

We define $U_S^p\left(I; L^2 \right)$, $V_S^p\left(I; L^2 \right)$ spaces to be the set of all functions $u:I \to L^2$ such that the following $U_S^p\left(I; L^2 \right)$-norm and $V_S^p\left(I; L^2 \right)$-norm are finite:
\begin{align*}
\Vert u \Vert_{U_S^p\left(I; L^2 \right)}:=&\Vert S(-t)u\Vert_{U^p\left(I;L^2 \right)} \quad \text{and} \quad \Vert u \Vert_{V_S^p\left(I; L^2 \right)}:=\Vert S(-t)u\Vert_{V^p\left(I; L^2 \right)},
\end{align*}
where $S(t)=e^{it\partial_x^4}$ denotes the linear propagator for $\eqref{eqn:fourth order NLS}$. Also we use the notation $U_S^p=U_S^p\left(I;L^2\right)$ and $V_S^p=V_S^p\left(I;L^2\right)$ if there is no confusion.
\begin{remark}
Observe that $U^p_S$ is the atomic space, where atoms are piecewise solutions to the linear equation
\begin{align*}
u=\sum\limits_k \chi_{  [t_{k-1},t_k  ) }e^{it\partial_x^4}(t)\phi_{k-1}, \quad \sum\limits_k \Vert \phi_{k-1}\Vert_{L^2\left(\mathbb{R}\right)}^p=1.
\end{align*}
\end{remark}
We denote by $DU_S^p$ the space of functions
\begin{align*}
DU_S^p=\left\{ \left(i\partial_t-\partial_x^4  \right)u; u\in U_S^p          \right\}
\end{align*}
with the induced norm. Then we have the trivial bound
\begin{align}\label{eqn:trivial bound}
\Vert u \Vert_{U_S^p} \lesssim \Vert u_0 \Vert_{L_x^2}+\Vert \left(i\partial_t-\partial_x^4 \right)u \Vert_{DU_S^2}.
\end{align}   
Moreover we have the duality relations
\begin{align*}
\left(DU_S^p\right)^*=V_S^{p'}, \quad 1<p<\infty.
\end{align*}
More precisely, given $\phi\in V_{S}^{p'} $, the mapping $f\to \int \langle f, \phi \rangle_{L^2} \,dt $ belongs $\left( DU_S^p     \right)^*$ and this identification is a surjective isometry. In fact, the spaces $DU_S^p$ and $DV_S^p$ are characterized as the spaces  for which the following norms are finite:
\begin{align}
\Vert f \Vert_{DU_S^p}=&\sup\left\{\int \langle f,\phi \rangle_{L^2} \,dt : \Vert \phi \Vert_{V_S^{p'}}\leq 1, \phi \in C_c^{\infty}    \right\}\label{eqn:Du^p duality}\\
\Vert f \Vert_{DV_S^p}=&\sup\left\{\int \langle f,\phi \rangle_{L^2} \,dt : \Vert \phi \Vert_{U_S^q}\leq 1, \phi \in C_c^{\infty}    \right\}\label{eqn:DV^p duality}.
\end{align}
More specifically, see for instance \cite{KT2007},\cite{KT2018}.

There is another choice for estimating the solution to $\eqref{eqn:fourth order NLS}$. The Bourgain's $X^{s,b}$ spaces is defined by 
\begin{align*}
\Vert u \Vert_{X^{s,b}}^2=\int \vert \widehat{ u}(\tau,\xi) \vert^2 \langle \xi \rangle^{2s}\langle \tau-\xi^4 \rangle^{2b} \,d\xi d\tau,
\end{align*}
where $\langle \cdot \rangle=\left(1+\vert \cdot \vert^2\right)^{\frac{1}{2}}$.
The space $X^{s,b}$ turns out to be very useful in the study of low regularity theory. But for $b=\frac{1}{2}$, logarithmic divergences happen in several estimates. 
To deal with this issues we consider dyadic decompositions with respect to the modulation $\tau-\xi^4$. This leads to the additional homogeneous Besov type norms
\begin{align*}
\Vert u \Vert_{\dot{X}^{s,\frac{1}{2},1 }  }=&\sum\limits_M \left( \int_{\vert \tau-\xi^4 \vert\approx M }  \vert \widehat{u}\left(\tau,\xi\right)\vert^2 \xi^{2s} \vert \tau-\xi^4 \vert \,d\xi d\tau  \right)^{\frac{1}{2}}\\
\Vert u \Vert_{\dot{X}^{s,\frac{1}{2},\infty }}=& \sup\limits_{M} \left(\int_{\vert \tau-\xi^4 \vert \approx M } \vert \widehat{ u}(\tau,\xi) \vert^2 \xi^{2s} \vert \tau-\xi^4 \vert \,d\xi d\tau  \right)^{\frac{1}{2}}.
\end{align*}
These homogeneous Besov type spaces are closely related to the spaces $U_S^2$ and $V_S^2$. Combining the embedding $V^2 \hookrightarrow \dot{B}^{\frac{1}{2}}_{2,\infty}$ with duality we have
\begin{align*}
\dot{X}^{0,\frac{1}{2},1}\hookrightarrow U^2_S \hookrightarrow V^2_S \hookrightarrow \dot{X}^{0,\frac{1}{2},\infty}.
\end{align*} 
In the following we use a Littlewood-Paley decomposition with respect to the modulation $\tau-\xi^4$ as well as a spatial Littlewood-Paley decomposition
\begin{align*}
1=\sum\limits_{N\geq 1 }P_N, \quad\quad 1=\sum\limits_{N\geq 1 }Q_N
\end{align*}  
Both decompositions are inhomogeneous. It is easy to see that we have the uniform boundedness properties
\begin{align*}
P_N:U_S^p \to U_S^p, \quad\quad Q_N:U_S^p \to U_S^p
\end{align*}    
and similarly for $V_S^p$. Moreover $U_S^p$ and $V_S^p$ spaces behave well with respect to sharp time cut off. If $I$ is a time interval, then we have
\begin{align*}
\chi_I:U_S^p \to U_S^p ,\quad\quad \chi_I:V_S^p \to V_S^p
\end{align*} 
with uniform bounds with respect to $I$.

We define an energy space with a standard energy norm
\begin{align*}
\Vert u \Vert_{\ell_N^2L_t^{\infty}H_x^s }^2:=\sum\limits_{N\geq 1}N^{2s} \Vert u_N \Vert_{L^\infty_t L_x^2}^2, 
\end{align*}
where we sum over all dyadic numbers $\geq 1$ with the obvious modification at $N=1$.
Note that $\Vert u \Vert_{L^\infty_t H^s_x }\leq \Vert u \Vert_{\ell_N^2 L^\infty_t H^s }$, but the converse is not true.

To estimate the solutions to $\eqref{eqn:fourth order NLS}$ we define the space $X^s$ with the norm
\begin{align}
\Vert u \Vert_{X^s}^2:=\sum\limits_{N\geq 1}N^{2s}\sup\limits_{\vert I \vert=N^{4s+2}} \Vert \chi_I u_N \Vert_{U_S^2}^2  
\end{align}
where we sum over all dyadic numbers $\geq 1$ with the obvious modification at $N=1$ and the supremum is taken over all subintervals $I\subset [0,1]$ of length $N^{4s+2}$.

To measure the regularity of the nonlinear term we define the space $Y^s$ with the norm
\begin{align}
\Vert f \Vert_{Y^s}^2:=\sum\limits_{N\geq 1}N^{2s}\sup\limits_{\vert I \vert=N^{4s+2}} \Vert \chi_I f_N \Vert_{DU_S^2}^2  
\end{align}
where we sum over all dyadic numbers $\geq 1$ with the obvious modification at $N=1$ and the supremum is taken over all subintervals $I\subset [0,1]$ of length $N^{4s+2}$.

\begin{prop}\label{lem:nonlinear estimate}
Let $u$ be a solution to $i\partial_t u -\partial_x^4u=f$ on $I$. Then we have
\begin{align*}
\Vert u \Vert_{X^s}\lesssim \Vert u \Vert_{\ell_N^2 L^{\infty}_{t}H^s }+\Vert f \Vert_{Y^s}.
\end{align*} 	
This proposition follows from \eqref{eqn:trivial bound}
\begin{align*}
\Vert u \Vert_{U_S^p}\lesssim \Vert u_0 \Vert_{L_x^2}+\Vert f \Vert_{DU_S^2}.
\end{align*}
More specifically, see \cite{CHT2012},\cite{KT2007}

\end{prop}

\section{Strichartz, local smoothing and bilinear Strichartz estimate}\label{sec:bilinear estimate}

In this section, we collect the standard linear and bilinear estimates.

\begin{lemma}[\cite{Seong2018}]\label{eqn:Strichartz estimate}
 For any $\alpha \in [0,1]$ , we call a pair $\left(q,r,\alpha \right)$ admissible exponents if $ r \geq 2, q\geq\frac{8}{(1+\alpha)}$ and $\frac{4}{q}+\frac{1+\alpha}{r}=\frac{1+\alpha}{2}$. Then for any admissible exponents $\left(q,r,\alpha\right)$, we have
\begin{align}\label{eqn:general strichartz estiamtes}
\left\| D^{\frac{\alpha}{2}\left(1-\frac{2}{r}\right)} e^{it\partial_x^4}u_0  \right\|_{L_t^qL_x^r\left(\mathbb{R}\times\mathbb{R}\right)}\lesssim_{q,r} & \left\| u_0 \right\|_{L_x^2\left(\mathbb{R}\right)}
\end{align} 
In particular, we have
\begin{align}
\Vert D^{\frac{1}{2}} e^{it\partial_x^4} u_0 \Vert_{L_t^4L_x^\infty\left(\mathbb{R}\times \mathbb{R}\right)}\lesssim&\Vert u_0 \Vert_{L_x^2\left(\mathbb{R}\right)},\label{eqn:D^{1/2} strichartz.}
\\
\Vert D^{\frac{1}{3}} e^{it\partial_x^4} u_0 \Vert_{L_t^6L_x^6\left(\mathbb{R}\times \mathbb{R}\right)}\lesssim&\Vert u_0 \Vert_{L_x^2\left(\mathbb{R}\right)}\label{eqn:D^{1/3} Strichartz},\\
\Vert  e^{it\partial_x^4} u_0 \Vert_{L_t^\infty L_x^2\left(\mathbb{R}\times \mathbb{R}\right)}\lesssim&\Vert u_0 \Vert_{L_x^2\left(\mathbb{R}\right)}\label{eqn:trivial Strichartz}.
\end{align}
\end{lemma}

\begin{corollary}
Let $I=[a,b)$ be an interval. Then for any admissible pair $(q,r,\alpha)$, we have 
\begin{align}\label{eqn:Stricharz L-P piece N}
\Vert P_{N}u \Vert_{L_t^qL_x^r\left(I\times \mathbb{R}\right)}\lesssim N^{-\frac{\alpha}{2}\left(1-\frac{2}{r} \right)} \Vert \chi_I u \Vert_{U_S^q}.
\end{align}	
Moreover, we have the dual estimate for $q>2$:
\begin{align}\label{eqn:Stricharz dual L-P piece N}
\Vert P_Nu \Vert_{DU^2_S\left(I;L^2\right)}\lesssim N^{-\frac{\alpha}{2}\left(1-\frac{2}{r} \right)} \Vert u \Vert_{L_t^{q'}L_x^{r'}\left(I\times\mathbb{R}\right)}.
\end{align}
\end{corollary}

\begin{proof}
We may assume $I=\mathbb{R}$ since $\chi_I$ can be inserted. It is enough to consider a $U_S^q$-atom $u$ :
\begin{align}
u(t,x)=\sum\limits_{k=1}^K\chi_{[t_{k-1},t_k)}(t)S(t)\phi_{k-1}(x),\;\sum\limits_{k=1}^K \Vert \phi_{k-1}\Vert_{L_x^2}^q\leq 1, \; \text{and}\; \phi_0=0.
\end{align}	
and show that 
\begin{align}
\Vert P_{N}u \Vert_{L_t^qL_x^r}\lesssim N^{-\frac{\alpha}{2}\left(1-\frac{2}{r} \right)}.
\end{align}
By using Strichartz estimates \eqref{eqn:general strichartz estiamtes} we have
\begin{align*}
\Vert P_{N}u \Vert_{L_t^qL_x^r}^q=&\sum\limits_{k=1}^K \Vert \chi_{[t_{k-1},t_k) }(t) P_NS(t) \phi_{k-1} \Vert_{L_t^qL_x^r}^q\\
\lesssim & N^{-\frac{\alpha}{2}\left(1-\frac{2}{r} \right)q} \sum\limits_{k=1}^K \Vert \phi_{k-1}\Vert_{L_x^2}^q\\
=& N^{-\frac{\alpha}{2}\left(1-\frac{2}{r} \right)q}.
\end{align*}
Now we prove dual estimate. Recall that $\left(DU^2\left(I;L^2\right) \right)^*=V^2\left(I;L^2\right)$. Therefore, we have
\begin{align*}
\Vert P_Nu \Vert_{DU^2_S\left( I;L^2\right)}=&\sup\limits_{\Vert v \Vert_{V_S^2\left(I;L^2\right)  }\leq 1 }\left|  \int_I\int_\mathbb{R} P_Nu \overline{v}\,dxdt    \right|\\
\lesssim& \sup\limits_{\Vert v \Vert_{V_S^2\left(I;L^2\right)  }\leq 1 } \Vert u \Vert_{L_t^{q'}L_x^{r'}\left(I\times\mathbb{R}\right)} \Vert P_N v \Vert_{L_t^{q}L_x^r\left(I\times \mathbb{R}\right)}.
\end{align*}	
By applying $\eqref{eqn:Stricharz L-P piece N}$ and Lemma \ref{lem:U^p ,V^p embedding lemma}, we have
\begin{align*}
\Vert P_N v \Vert_{L_t^qL_x^r\left(I\times \mathbb{R}\right)}\lesssim& N^{-\frac{\alpha}{2}\left(1-\frac{2}{r} \right)}\Vert \chi_Iv \Vert_{U_s^qL^2}\\\lesssim& N^{-\frac{\alpha}{2}\left(1-\frac{2}{r} \right)} \Vert \chi_Iv \Vert_{V_S^2L^2}
\end{align*}
which complete the proof.
\end{proof}

\begin{lemma}[Local smoothing, maximal function estimates \cite{KPV1991}]
\begin{align}
\Vert D^{\frac{3}{2}} e^{it\partial_x^4} u_0 \Vert_{L_x^\infty L_t^2}\lesssim& \Vert u_0 \Vert_{L_x^2}\\
\Vert D^{-\frac{1}{4}} e^{it\partial_x^4} u_0 \Vert_{L_x^4 L_t^\infty}\lesssim& \Vert u_0 \Vert_{L_x^2}\label{eqn:L_tL_x maximal function estimate}
\end{align}	
\end{lemma}
\begin{proof}
In the case of local smoothing estimate, the proof is the same as Schr\"odinger case. In fact, it is reducible to using Plancherel theorem in $L_t^2$. For the proof of the maximal function estimate, see Theorem 2.5 in \cite{KPV1991}.
\end{proof}
\begin{corollary}
Let $I=[a,b)$ be an interval. Then we have
\begin{align}
\Vert P_Nu \Vert_{L_x^\infty L_t^2\left(I\times \mathbb{R}\right)}\lesssim& N^{-\frac{3}{2}}\Vert \chi_I u \Vert_{U_S^2}\label{eqn:local smoothing estimate U^2},\\
\Vert P_Nu \Vert_{L_x^4  L_t^\infty    \left(I\times \mathbb{R}\right)}\lesssim& N^{\frac{1}{4}}\Vert \chi_I u \Vert_{U_S^4}\label{eqn:maximal function estimate U^4}.
\end{align}
\end{corollary}
\begin{proof}
As we proceed in the proof of Corollary 3.2, it suffices to consider $U_S^2,U_S^4$ atoms respectively. 
For more details, see Proposition 2.19 in \cite{HHK2009}.
\end{proof}

\begin{prop}[Bilinear estimate]\label{prop:Lesbegue space bilinear}
Let $A_1,A_2\subset \mathbb{R}$  such that
\begin{align*}
\vert \xi_1^3-\xi_2^3 \vert=\vert \xi_1-\xi_2\vert \vert \xi_1^2+\xi_1 \xi_2 +\xi_2^2 \vert \approx N N_{\max}^2, 
\end{align*}
for all $\xi_1 \in A_1$,$\xi_2 \in A_2$ and $N_{\max}=\max\left\{\vert \xi_1 \vert,\vert \xi_2 \vert \right\} $. 	
We define the projection operators $P_{A_j}$ as the Fourier multiplier operators $\widehat{P_{A_j}\psi}(\xi)=\chi_{A_j}\widehat{\psi}(\xi)$ for a function $\psi$.
Then we have
\begin{align}
\Vert P_{A_1}e^{it\partial_x^4}\phi_1 P_{A_2}e^{it\partial_x^4}\phi_2 \Vert_{L_t^2L_x^2}\lesssim N^{-\frac{1}{2}}N_{\max}^{-1} \Vert P_{A_1}\phi_1 \Vert_{L_x^2} \Vert P_{A_2}\phi_2 \Vert_{L_x^2}. 
\end{align}
\end{prop}

\begin{proof} 
    Note that by duality, 
	\begin{align*}
	&\left\| P_{A_1}e^{it\partial_x^4}\phi_1  P_{A_2}e^{it\partial_x^4}\phi_{2} \right\|_{L_{t,x}^2\left(\mathbb{R}\times \mathbb{R}\right)}\\
	=&\left\|\int_{\mathbb{R}}e^{it\xi_1^4}\widehat{P_{A_1}\phi_{1}}\left(\xi_1\right) e^{it\left(\xi-\xi_1\right)^4}\widehat{P_{A_2}\phi_2}\left(\xi-\xi_1\right) \,d\xi_1 \right\|_{L_{t,\xi}^2\left(\mathbb{R}\times \mathbb{R}\right)}\\
	=&\sup\limits_{\left\|\psi \right\|_{L_{t,\xi}}^2=1} \left|\int_{\mathbb{R}\times\mathbb{R}} \int_{\mathbb{R}}e^{it\xi_1^4}\widehat{P_{A_1}\phi_{1}}\left(\xi_1\right) e^{it\left(\xi-\xi_1\right)^4}\widehat{P_{A_2}\phi_2}\left(\xi-\xi_1\right)\psi\left(t,\xi
	\right) \,d\xi_1 d\xi dt \right|.
	\end{align*}
	Hence, it suffcies to consider the integral 
	\begin{align*}
	&\left|\int_{\mathbb{R}\times\mathbb{R}} \int_{\mathbb{R}}e^{it\xi_1^4}\widehat{P_{A_1}\phi_1}\left(\xi_1\right) e^{it\left(\xi-\xi_1\right)^4}\widehat{P_{A_2}\phi_2}\left(\xi-\xi_1\right)\psi\left(t,\xi
	\right) \,d\xi_1 d\xi dt \right|\\
	=& \left|\int_{\mathbb{R}\times\mathbb{R}} \int_{\mathbb{R}}e^{it\xi_1^4}\widehat{P_{A_1}\phi_1}\left(\xi_1\right) e^{it \xi_2^4}\widehat{P_{A_2}\phi_2}\left(\xi_2\right)\psi\left(t,\xi_1+\xi_2
	\right) \,d\xi_1 d\xi_2 dt \right|\\
	=& \left|\int_{\mathbb{R}} \int_{\mathbb{R}}\widehat{P_{A_1}\phi_1}\left(\xi_1\right)\widehat{P_{A_2}\phi_2}\left(\xi_2\right)\mathcal{F}_t\psi\left(\xi_1^4+\xi_2^4,\xi_1+\xi_2 \right) \,d\xi_1 d\xi_2  \right|.
	\end{align*}
	We consider the change of variable $\left(\xi_1,\xi_2\right) \to \left(\eta_1(\xi_1,\xi_2),\eta_2(\xi_1,\xi_2)\right)=\left(\xi_1+\xi_2,\xi_1^4+\xi_2^4 \right)$ with the Jacobian $\left| J \right|=4\left| \xi_1^3-\xi_2^3\right|\approx N N_{\max}^2$.	
	Hence, 
	\begin{align*}
	\left| J \right|d\xi_1 d\xi_2=d\eta_1 d\eta_2 \quad \text{or} \quad d\xi_1 d\xi_2    = \left| J \right|^{-1} d\eta_1 d\eta_2.
	\end{align*}
	So, by H\"older's inequality, the Plancherel theorem and again change of variable, we have
	\begin{align*}
	&\left|\int_{\mathbb{R}} \int_{\mathbb{R}}\widehat{P_{A_1}\phi_1}\left(\xi_1\right)\widehat{P_{A_2}\phi_2}\left(\xi_2\right)\mathcal{F}_t\psi\left(\xi_1^4+\xi_2^4,\xi_1+\xi_2 \right) \,d\xi_1 d\xi_2  \right|\\
	\lesssim& \left\| \mathcal{F}_t \psi \right\|_{L_{\eta_1,\eta_2}^2} \left|\int_{ \eta_1\in \mathbb{R}} \int_{\eta_2\in\mathbb{R}}\left|\widehat{P_{A_1}\phi_1}\left(\eta_1,\eta_2 \right)\widehat{P_{A_2}\phi_2}\left(\eta_1,\eta_2\right)\right|^{2}  \left| J\right|^{-2} \,d\eta_1 d\eta_2 \right|^{\frac{1}{2}}\\
	\lesssim &  N^{-\frac{1}{2}}N_{\max}^{-1} \left|\int_{ \xi_1\in \mathbb{R}} \int_{\xi_2\in\mathbb{R}}\left|\widehat{P_{A_1}\phi_1}\left(\xi_1 \right)\widehat{P_{A_2}\phi_2}\left(\xi_2\right)\right|^{2}  \,d\xi_1 d\xi_2 \right|^{\frac{1}{2}}\\
	\lesssim & N^{-\frac{1}{2}}N_{\max}^{-1} \left\|P_{A_1}\phi_{1} \right\|_{L_x^2\left(\mathbb{R}\right)} \left\|P_{A_2}\phi_{2} \right\|_{L_x^2\left(\mathbb{R}\right)}.
	\end{align*}	
\end{proof}

\begin{corollary}[bilinear Strichartz estimates]  
Let $N_1\ll N_2$. Then we have
\begin{align}
\Vert P_{N_1}u_1 P_{N_2}u_2 \Vert_{L_t^2L_x^2\left(I\times \mathbb{R}\right)}\lesssim& N_2^{-\frac{3}{2}} \Vert \chi_I P_{N_1}u_1 \Vert_{U_S^2L^2}\Vert \chi_I P_{N_1}u_1 \Vert_{U_S^2L^2},\label{eqn:bilinear not log}\\
\intertext{and}
\Vert P_{N_1}u_1 P_{N_2}u_2 \Vert_{L_t^2L_x^2\left(I\times \mathbb{R}\right)}\lesssim& N_2^{-\frac{3}{2}} \left(\ln N_2+1 \right)^2 \Vert \chi_I P_{N_1}u_1 \Vert_{V_S^2L^2}\Vert \chi_I P_{N_1}u_1 \Vert_{V_S^2L^2}.\label{eqn:log interpolation}
\end{align}		
\end{corollary}
\begin{proof}
In fact \eqref{eqn:bilinear not log} is the result of the transference principle. For more details, see Proposition 2.19 in \cite{HHK2009}. The proof of estimate $\eqref{eqn:log interpolation}$ follows from the argument in [\cite{HHK2009}, Proposition 2.20, Corollary 2.21.] 	
\end{proof}

  
\section{trilinear estimate}\label{sec:trilinear estimate}
In this section, we prove the following trilinear estimates below $s<-1/2$. As we mentiond above, we use both dispersive smoothing effects (bilinear Strichartz estimates $\eqref{eqn:bilinear not log}, \eqref{eqn:log interpolation}$) and short time structure. Our method is inspired by Koch-Tataru \cite{KT2007}. 
\begin{prop}[Trilinear estimate]\label{prop:trilinear estimate} Let $-3/4<s<-1/2$. Then we have
\begin{align}\label{eqn:trilinear estimate}
\Vert u_1 \overline{u_2}u_3 \Vert_{Y^s}\lesssim \Vert u_1 \Vert_{X^s} \Vert u_2 \Vert_{X^s} \Vert u_3 \Vert_{X^s}.
\end{align}	
\end{prop}

\begin{proof}
We estimate the nonlinearity $\vert u \vert^2 u $ at frequency $N$ in a $N^{4s+2}$ time interval $I$. We also consider a full dyadic decomposition of each of the factors.
\begin{align*}
\Vert u_1 \overline{u_2} u_3 \Vert_{Y^s}=&
\left( \sum\limits_{N\geq 1} N^{2s}\sup\limits_{\vert I \vert=N^{4s+2}} \Vert \chi_I P_N\left( u_1 \overline{u_2} u_3\right) \Vert_{DU_S^2 }^2\right)^{\frac{1}{2}}\\
\lesssim &  \sum\limits_{N\geq 1} N^{s}\sup\limits_{\vert I \vert=N^{4s+2}} \Vert \chi_I P_N\left( u_1 \overline{u_2} u_3\right) \Vert_{DU_S^2 } 
\intertext{and}
\Vert \chi_I P_N \left(u_1 \overline{u_2} u_{3} \right) \Vert_{DU^2_S}\lesssim& \sum\limits_{1\leq N_1,N_2,N_3}       \Vert \chi_I P_N \left(u_{N_1} \overline{u_{N_2}} u_{N_3} \right) \Vert_{DU^2_S}
\end{align*}
Therefore, we need to show that for an interval $\vert I \vert=N^{4s+2}$
\begin{align*}
\Vert \chi_I P_{N} \left(u_{N_1}u_{N_2}u_{N_3} \right) \Vert_{DU_S^2 }\lesssim & N_1^{\alpha_1}N_2^{\alpha_2}N_3^{\alpha_3}N^{\alpha}   \prod\limits_{j=1}^3\sup\limits_{\vert I_j \vert=N_j^{4s+2}} \Vert \chi_{I_j}u_{N_j} \Vert_{U_S^2}
\end{align*}  
where $N_1^{\alpha_1}N_2^{\alpha_2}N_3^{\alpha_3}N^{\alpha}$ have summability with respect to $N_1,N_2,N_3,N\geq 1$.
We denote 
\begin{align*}
\left\{N,N_1,N_2,N_3 \right\}=\left\{M,M_1,M_2,M_3 \right\}, \quad M_1\leq M_2 \leq M_3, \; N=M.
\end{align*}
Note that the two largest frequencies must be comparable. Hence, we investigate the following cases of interactions:
	
\noindent \textbf{Case 1.} $M_1\leq M_2 \leq M_3 \approx M $ .\\
\textbf{Case 2.} $M_1\lesssim M \ll M_2\approx M_3$.\\
\textbf{Case 3.} $M\ll M_1\leq M_2\approx M_3$.\\

\textbf{Case 1.} $M_1 \leq M_2 \leq M_3 \approx  M $.
In this case, we observe that $\vert I_j \vert \geq \vert I \vert$ for $j=1,2,3$. Hence there is no interval summation loss. 
At the case 1, there is no role of the complex conjugate. Therefore we drop the complex conjugate sign.   
By using the dual estimates $\eqref{eqn:Stricharz dual L-P piece N}$, Strichartz estimate $\eqref{eqn:D^{1/3} Strichartz}$ and embedding $U_S^2\hookrightarrow U_S^6$, we have
\begin{equation}
\begin{split}
\label{short time gain}
\Vert \chi_I P_M\left(u_{M_1}u_{M_2}u_{M_3}   \right) \Vert_{DU^2_S }\lesssim& \Vert u_{M_1} u_{M_2} u_{M_3} \Vert_{L_t^1 L_x^2\left(I\times \mathbb{R}\right)}\\
\lesssim & \vert I \vert^{1/2} \Vert u_{M_1}u_{M_2}u_{M_3} \Vert_{L_t^2L_x^2\left(I\times \mathbb{R}\right)}\\
=& M^{2s+1} \Vert u_{M_1}u_{M_2}u_{M_3} \Vert_{L_t^2L_x^2\left(I\times \mathbb{R}\right)}\\
\lesssim & M^{2s+1} \Vert u_{M_1} \Vert_{L_t^6L_x^6\left(I\times \mathbb{R}\right)} \Vert u_{M_2} \Vert_{L_t^6L_x^6\left(I\times \mathbb{R}\right)} \Vert u_{M_3} \Vert_{L_t^6L_x^6\left(I\times \mathbb{R}\right)}\\
\lesssim & M^{2s+1} M_1^{-\frac{1}{3}} M_2^{-\frac{1}{3}}M_3^{-\frac{1}{3}}\prod\limits_{j=1}^3\Vert \chi_{I}u_{N_j} \Vert_{U_S^6}\\
\lesssim & M^{2s+1} M_1^{-\frac{1}{3}} M_2^{-\frac{1}{3}}M_3^{-\frac{1}{3}}\prod\limits_{j=1}^3\Vert \chi_{I}u_{M_j} \Vert_{U_S^2}.
\end{split}
\end{equation} 
By combining above calculations, we need to focus on the following summation
\begin{align*}
&\sum\limits_{M_1\leq M_2 \leq M_3 \approx M}M^s\sup\limits_{\vert I \vert=M^{4s+2}} \Vert \chi_I P_M\left( u_{M_1}u_{M_2}u_{M_3}\right) \Vert_{DU_S^2}\\
\lesssim& \sum\limits_{M_1\leq M_2 \leq M_3 \approx M} M^{2s+1} M_{1}^{-\frac{1}{3}-s}M_2^{-\frac{1}{3}-s}M_3^{-\frac{1}{3}} \prod_{j=1}^3 \Vert u_{M_j} \Vert_{X^s}\\
\lesssim&\;  \Vert u_1 \Vert_{X^s}\sum\limits_{M_2\leq M_3}M_2^{-2s-\frac{2}{3}} M_3^{2s+\frac{2}{3}} \Vert u_{M_2} \Vert_{X^s} \Vert u_{M_3} \Vert_{X^s}. 
\end{align*}
Therefore, by using Schur's test, we obtain the desired result.
Observe that in this interaction, trilinear estimate is satisfied for all $s<-1/2$.\\

\textbf{Case 2.} $M_1\lesssim M \ll M_2\approx M_3$.
Similarly in case 1 there is no role of complex conjugate and hence we drop the complex conjugate sign. Observe that the $u_{M_2},u_{M_3}$ should be estimated in norm with timescale $M_3^{4s+2}$. We divide $I$ into $\vert I\vert /  \vert J \vert=\left(\frac{M_3}{M}\right)^{-4s-2}\gg 1$ intervals of size $\vert J \vert=M_3^{4s+2}$. By applying the duality $\eqref{eqn:Du^p duality}$, we have
\begin{align*}
\Vert \chi_I P_M\left(u_{M_1}u_{M_2}u_{M_3}\right) \Vert_{DU^2_S }=\sup\limits_{\Vert u \Vert_{V_S^2}=1} \left| \int_I \int_{\mathbb{R}} u_{M_1}u_{M_2}u_{M_3}P_Mu\, dxdt \right|.
\end{align*}
By using the bilinear Strichartz estimates $\eqref{eqn:bilinear not log}$ ,$\eqref{eqn:log interpolation}$, we estimate
\begin{align*}
\left|  \int_I\int_\mathbb{R} u_{M_1}u_{M_2}u_{M_3}u_{M}\,dxdt  \right|\lesssim& \left(\frac{M_3}{M}\right)^{-4s-2}\sup\limits_{\substack{J\subset I \\ \vert J \vert=M_3^{4s+2} }} \left | \int_J\int_\mathbb{R} u_{M_1}u_{M_2}u_{M_3}u_{M}\,dxdt \right|\\
\lesssim& \left(\frac{M_3}{M}  \right)^{-4s-2}\sup\limits_{\substack{J\subset I \\ \vert J \vert=M_3^{4s+2}}}\Vert \chi_Ju_{M_1}u_{M_2} \Vert_{L_t^2L_x^2}\Vert \chi_J u_M u_{M_3} \Vert_{L_t^2L_x^2}\\
\lesssim & \left(\frac{M_3}{M}\right)^{-4s-2}M_3^{-3}\left(\log M_3\right)^2 \\
&\times \sup\limits_{ \substack{J\subset I \\ \vert J \vert=M_3^{4s+2}} } \Vert \chi_J u_{M_1} \Vert_{U_S^2} \Vert \chi_J u_{M_2} \Vert_{U_S^2} \Vert \chi_J u_{M_3} \Vert_{V_S^2} \Vert \chi_J u_{M} \Vert_{V_S^2}.
\end{align*}
By combining above calculation, we focus on the following summation:
\begin{align*}
&\sum\limits_{M_1\lesssim M \ll M_2\approx M_3 } M^s\sup\limits_{\vert I \vert=M^{4s+2}} \Vert \chi_I P_M\left( u_{M_1}u_{M_2}u_{M_3}\right) \Vert_{DU_S^2}\\
&\lesssim \sum\limits_{M_1\lesssim M\ll M_2\approx M_3}M_3^{-6s-5+}M^{5s+2}M_1^{-s}\Vert u_{M_1} \Vert_{X^s} \Vert u_{M_2} \Vert_{X^s} \Vert u_{M_3} \Vert_{X^s}.
\end{align*}  
We can deal with above summation by using Schwarz inequality if $s> -\frac{5}{6}$.\\

\textbf{Case 3.} $M\ll M_1\leq M_2\approx M_3$.
Case 3 is the worst case in terms of the short time structure  because the interval summation loss is the largest. 
However, because case 3 is also a nonresonant interaction unlike the other two cases above, it is necessary to take advantage of nonresonant interaction to weaken this interval summation loss. 

To demonstrate the case 3, we need the following lemma

\begin{lemma}\label{lem:improved trilinear}
Let $-\frac{3}{4}< s < -\frac{1}{2}$ and $I$ be a time interval with $\vert I \vert=N^{4s+2}$.
Define 
\begin{align*}
f=&\sum\limits_{\substack{N_1,N_2,N_3:\\N_1,N_2,N_3\gg N}}\chi_I P_N \left( u_{N_1}\overline{u_{N_2}}u_{N_3}\right)\\
=&\sum\limits_{\substack{N_1,N_2,N_3:\\N_1,N_2,N_3\gg N}}\sum\limits_{\substack{J\subset I \\ \vert J \vert=N_{\max}^{4s+2}}} P_N \left(\chi_J u_{N_1} \chi_J \overline{u_{N_2} } \chi_J u_{N_3}   \right), \quad N_{\max}=\max\left\{ N_1,N_2,N_3\right\}.
\end{align*}
Here the $J$ summation is understood to be over a partition of $I$ into intervals $J$ of the indicated size.
Then, we have the estimates
\begin{align*}
\Vert Q_{\geq N^4 } f \Vert_{\dot{X}^{0,-\frac{1}{2},1}} \lesssim N^{-3-3s} \Vert u \Vert_{X^s}^3
\intertext{and}
\Vert Q_{\leq N^4} f \Vert_{L_t^1L_x^2+N^{-\frac{1}{4}  }  L_x^{\frac{4}{3}}L_t^1      }\lesssim N^{-3-3s } \Vert u \Vert_{X^s }^3. 
\end{align*}	
\end{lemma} 

\begin{remark}
	In fact, the same estimates hold if we replace $\overline{u_{N_2}}$ by $u_{N_2}$ and this case become easier. 	
\end{remark}

\begin{remark}\label{rem:-5/7 trilinear}
By using only the bilinear smoothing effect without Lemma \ref{lem:improved trilinear}, we can prove the tirlinear estimate up to $s>-\frac{5}{7}$. 

In case 3 $M\ll M_1\leq M_2\approx M_3$,
we consider two subcase $M\ll M_1\ll M_2\approx M_3$ and
 $M\ll M_1\approx M_2\approx M_3$. 

 Subcase 3.a. $M\ll M_1\ll M_2\approx M_3$. By using the same method as in case 2, we need to focus on the following summation
\begin{align*}
&\sum\limits_{M\ll M_1\leq M_2\approx M_3}M_3^{-6s-5+}M^{5s+2}M_1^{-s} \Vert u_{M_1}\Vert_{X^s} \Vert u_{M_2} \Vert_{X^s} \Vert u_{M_3} \Vert_{X^s}.
\end{align*} 
The summation is handled by using Schwarz inequality if $s> -\frac{5}{7}$. 

Subcase 3.b $ M\ll M_1\approx M_2 \approx M_3$. In order for the final output to be at frequency $M$, the two frequency $M_3$ factors must be $M_3$ separated. We may assume that $u_{M_2},u_{M_3}$ are $M_3$ separated. Therefore we obtain
\begin{align*}
\Vert \chi_J u_{M_2} u_{M_3} \Vert_{L_t^2L_x^2}\lesssim M_3^{-\frac{3}{2}}\Vert \chi_J u_{M_2} \Vert_{U_S^2} \Vert \chi_J u_{M_3} \Vert_{U_S^2}
\intertext{and}
\Vert \chi_J u_{M} u_{M_1} \Vert_{L_t^2L_x^2}\lesssim M_3^{-\frac{3}{2}}\log M_3\Vert \chi_J u_{M} \Vert_{V_S^2} \Vert \chi_J u_{M_1} \Vert_{V_S^2}.
\end{align*}
Therefore we can proceed as in Case 2. we focus on the following summation
\begin{align*}
&\sum\limits_{M\ll M_1\approx M_2\approx M_3}M_3^{-7s-5+}M^{5s+2} \Vert u_{M_1}\Vert_{X^s} \Vert u_{M_2} \Vert_{X^s} \Vert u_{M_3} \Vert_{X^s}.
\end{align*}
The summation is handled by using Schwarz inequality if $s> -\frac{5}{7}$.
\end{remark}

Observe that the embedding $\dot{X}^{0,\frac{1}{2},1} \hookrightarrow U^2_S$ implies the embedding $\dot{X}^{0,-\frac{1}{2},1}\hookrightarrow DU_S^2$. By using  $\dot{X}^{0,-\frac{1}{2},1}\hookrightarrow DU_S^2$, dual estimates $\eqref{eqn:Stricharz dual L-P piece N}$ and Lemma \ref{lem:improved trilinear} we have
\begin{align*}
N^s\Vert f \Vert_{DU_S^2}\lesssim N^{-3-2s}\Vert u_1 \Vert_{X^s} \Vert u_2 \Vert_{X^s} \Vert u_3 \Vert_{X^s}
\end{align*} 
Therefore the summation with respect to $N$ is easily handled. So we conclude the proof of Proposition \ref{prop:trilinear estimate}.
\end{proof}

\begin{proof}[Proof of Lemma \ref{lem:improved trilinear} ]
Let $\left(\tau_i,\xi_i \right)$ be the frequencies for each factor and let $\left(\tau,\xi  \right)$ be the resulting frequency. Then we have
\begin{align}\label{interaction}
\xi_1-\xi_2+\xi_3=\xi, \quad\quad \tau_1-\tau_2+\tau_3=\tau.
\end{align}
Under the relation $\left(\ref{interaction} \right)$, we have
\begin{align*}
\left(\tau_1-\xi_1^4 \right)-\left(\tau_2-\xi_2^4 \right)+\left(\tau_3-\xi_3^4 \right)-\left(\tau-\xi^4 \right)\\
=(\xi-\xi_1)(\xi-\xi_3)(\xi_1^2+\xi_2^2+\xi_3^2+\xi^2+2(\xi_1+\xi_3)^2).
\end{align*}
The size of frequencies $\left\{\xi, \xi_1,\xi_2,\xi_3  \right\}$ is $\left\{M,M_1,M_2,M_3  \right\}$ with $M\ll M_1 \lesssim M_2=M_3 $. Here we allow for a slight abuse of notation, as the highest $M_j$'s need not be equal but merely comparable. Therefore, we have
\begin{align*}
\vert \tau_1-\xi_1^4 \vert + \vert  \tau_2-\xi_2^4 \vert +\vert \tau_3-\xi_3^4 \vert +\vert \tau-\xi^4 \vert \gtrsim M_1 M_3 M_3^2 \quad \text{if $N_2=M_3$}
\intertext{and}
\vert \tau_1-\xi_1^4 \vert + \vert  \tau_2-\xi_2^4 \vert +\vert \tau_3-\xi_3^4 \vert +\vert \tau-\xi^4 \vert \gtrsim M_3 M_3 M_3^2 \quad \text{if $N_2=M_1$}
\end{align*}
Hence at least one modulation should be large. Therefore, to take advantage of this  modulation gain, we consider the following cases:\\

\noindent\textbf{Case 1.} All input factors have small modulation i.e., $\vert \tau_1-\xi_1^4 \vert, \vert  \tau_2-\xi_2^4 \vert,\vert \tau_3-\xi_3^4 \vert \ll M_1M_3^3$. Therefore, output has large modulation $\vert \tau-\xi^4 \vert \gtrsim M_1M_3^3$.\\
\textbf{Case 2.} There is at least one input which has large modulation i.e., $\vert \tau_1-\xi_1^4 \vert\gtrsim M_1M_3^3$ or $ \vert  \tau_2-\xi_2^4 \vert \gtrsim M_1M_3^3$ or $\vert \tau_3-\xi_3^4 \vert \gtrsim M_1M_3^3$.
Depending on which factor has the large modulation and on whether the conjugated factor $\overline{u}_{N_2}$ has lower frequency $M_1$ or not, we divide this case into six.\\

\textbf{Case 1.} This is 
when the input factors have small modulation and hence the output has large modulation. Depending on whether the conjugated factor has a lower frequency or not we divide this case into three.\\

\textbf{Subcase 1.a} First, we consider the first component of $f$ i.e., the conjugated factor $\overline{u}_{N_2}$ has the highest frequency $M_3$
\begin{align*}
f_1=&\sum\limits_{M\ll M_1\ll M_3} \sum\limits_{\substack{J\subset I \\ \vert J \vert=M_3^{4s+2} }} P_M \left(Q_{\ll M_1M_3^3 }\left(\chi_J  u_{M_3}\right) \overline{ Q_{\ll M_1M_3^3 } \left(\chi_J u_{M_3} \right)  }   Q_{\ll M_1M_3^3 }\left(\chi_J u_{M_1} \right)    \right)\\
=&\sum\limits_{M\ll M_1\ll M_3}\sum\limits_{\substack{ J\subset I \\ \vert J \vert=M_3^{4s+2} }} f_1^{M_1,M_3}.
\end{align*}
Since all input factors have small modulation $\ll M_1M_3^3$, we have
\begin{align*}
\vert \tau-\xi^4\vert=&\left\vert\left(\tau_1-\xi_1^4\right)-\left(\tau_2-\xi_2^4\right)+\left(\tau_3-\xi_3^4\right)+\left(\xi_1^4-\xi_2^4+\xi_3^4-\xi^4\right)\right\vert\\
\approx& M_1M_3^3.
\end{align*}
Hence $f_1^{M_1,M_3}$ is localized at modulation $M_1M_3^3 \gg M^4 $.

For the triple product $  v_{M_3} \overline{v_{M_3} } v_{M_1}   $, by using the energy bound for $v_{M_3}$ and the bilinear estimate for $v_{M_3}v_{M_1}$, we have
\begin{align*}
\Vert   v_{M_3} \overline{v_{M_3} } v_{M_1}    \Vert_{L_t^2L_x^1} \lesssim M_3^{-\frac{3}{2}} \Vert v_{M_3} \Vert_{U_{S}^2} \Vert v_{M_3} \Vert_{U_{S}^2} \Vert v_{M_1} \Vert_{U_{S}^2}.
\end{align*} 
Applying $P_M$ and the Bernstein's inequality, we obtain
\begin{align}\label{triple product}
\Vert P_M\left(  v_{M_3} \overline{v_{M_3} } v_{M_1}  \right)  \Vert_{L_t^2L_x^2} \lesssim M_3^{-\frac{3}{2}}M^{\frac{1}{2}} \Vert v_{M_3} \Vert_{U_{S}^2} \Vert v_{M_3} \Vert_{U_{S}^2} \Vert v_{M_1} \Vert_{U_{S}^2}.
\end{align}
In order to bound $f_1$, we need to consider the interval summation losss $M^{4s+2}M_3^{-4s-2}$.	
Therefore, by using $\eqref{triple product}$ and uniform boundedness property $Q_N : U_s^p \to U_s^p$, we obtain
\begin{align*} 
\sum\limits_{\substack{ J\subset I \\ \vert J \vert=M_3^{4s+2} }}\Vert f_1^{M_1,M_3} \Vert_{L_{t,x}^2} \lesssim M^{4s+2}M_3^{-4s-2}M_1^{-s}M_3^{-2s}M^{\frac{1}{2}}M_3^{-\frac{3}{2}} \Vert u_{M_3}\Vert_{X^s}\Vert u_{M_3}\Vert_{X^s}\Vert u_{M_1}\Vert_{X^s}. 
\end{align*}	
Since $f_1^{M_1,M_3}$ is localized at modulation $M_1M_3^3$, we have
\begin{align*}
\sum\limits_{\substack{ J\subset I \\ \vert J \vert=M_3^{4s+2} }}\Vert f_1^{M_1,M_3} \Vert_{\dot{X}^{0,-\frac{1}{2},1} } \lesssim M_3^{-6s-5}M_1^{-s-\frac{1}{2}}
M^{4s+\frac{5}{2}}  \Vert u_{M_3}\Vert_{X^s}\Vert u_{M_3}\Vert_{X^s}\Vert u_{M_1}\Vert_{X^s}.
\end{align*}
The summation with respect to the dyadic numbers $M_1$ and $M_3$ is handled if $s\geq -\frac{11}{14}$. 
As a result, we conclude
\begin{align*}
\Vert f_1 \Vert_{\dot{X}^{0,-\frac{1}{2},1} }\lesssim& \sum\limits_{M\ll M_1\ll M_3}\sum\limits_{\substack{ J\subset I \\ \vert J \vert=M_3^{4s+2} }} \Vert f_1^{M_1,M_3}\Vert_{\dot{X}^{0,-\frac{1}{2},1}} \\
\lesssim& M^{-3-3s}\Vert u \Vert_{X^s}^3.
\end{align*}

\textbf{Subcase 1.b} The second component of $f$ is that the conjugated factor $\overline{u}_{N_2}$ has lower frequency $M_1$
\begin{align*}
f_2=&\sum\limits_{M\ll M_1\ll M_3} \sum\limits_{\substack{J\subset I \\ \vert J \vert=M_3^{4s+2}}} P_M \left(Q_{\ll M_1M_3^3 }\left(\chi_J  u_{M_3}\right) \overline{ Q_{\ll M_1M_3^3 } \left(\chi_Ju_{M_1} \right)  }   Q_{\ll M_1M_3^3 }\left(\chi_J u_{M_3} \right)    \right)\\
=&\sum\limits_{M\ll M_1\ll M_3}\sum\limits_{\substack{ J\subset I \\ \vert J \vert=M_3^{4s+2} }}f_2^{M_1,M_3}.
\end{align*}
Then $f_2^{M_1,M_3}$ is localized at modulation $M_3^4\gg M^4$.
 
By using Bernstein, energy bound and bilinear estimate, we have the following $L_{t,x}^2$ bound:
\begin{align}\label{triplebound 2}
\Vert P_M\left(  v_{M_3} \overline{v_{M_1} } v_{M_3}  \right)  \Vert_{L_t^2L_x^2} \lesssim M_3^{-\frac{3}{2}}M^{\frac{1}{2}} \Vert v_{M_3} \Vert_{U_{S}^2} \Vert v_{M_3} \Vert_{U_{S}^2} \Vert v_{M_1} \Vert_{U_{S}^2}.
\end{align}
By considering the interval summation losses and using $\eqref{triplebound 2}$, we obtain
\begin{align*}
\sum\limits_{\substack{ J\subset I \\ \vert J \vert=M_3^{4s+2} }}\Vert f_2^{M_1,M_3} \Vert_{L_{t,x}^2}\lesssim M^{4s+2}M_3^{-4s-2}M_1^{-s}M_3^{-2s}M^{\frac{1}{2}}M_3^{-\frac{3}{2}} \Vert u_{M_3} \Vert_{X^s} \Vert u_{M_3} \Vert_{X^s} \Vert u_{M_1} \Vert_{X^s}.
\end{align*} 
Since $f_2^{M_1,M_3}$ has modulation $M_3^4$, we have
\begin{align*}
\sum\limits_{\substack{ J\subset I \\ \vert J \vert=M_3^{4s+2} }}\Vert f_2^{M_1,M_3} \Vert_{\dot{X}^{0,-\frac{1}{2},1}} \lesssim M_3^{-6s-\frac{11}{2}}M_1^{-s}M^{4s+\frac{5}{2}} \Vert u_{M_3} \Vert_{X^s} \Vert u_{M_3} \Vert_{X^s} \Vert u_{M_1} \Vert_{X^s}
\end{align*}
which is summable with respect to dyadic number $M_1,M_3$ if $s\geq -\frac{11}{14}$.
After summation, we obtain 
\begin{align*}
\Vert f_2 \Vert_{\dot{X}^{0,-\frac{1}{2},1} }\lesssim M^{-3-3s}\Vert u \Vert_{X^s}^3.
\end{align*}

\textbf{Subcase 1.c} The third component of $f$ is 
\begin{align*}
f_3=&\sum\limits_{M\ll M_1=M_3}\sum\limits_{\substack{J\subset I\\ \vert J \vert=M_3^{4s+2} }} P_M \left(Q_{\ll M_3^4 }\left(\chi_J  u_{M_3}\right) \overline{ Q_{\ll M_3^4 } \left(\chi_Ju_{M_3} \right)  }   Q_{\ll M_3^4 }\left(\chi_J u_{M_3} \right)    \right)\\
=&\sum\limits_{M\ll M_1 =M_3}\sum\limits_{\substack{ J\subset I \\ \vert J \vert=M_3^{4s+2} }}f_3^{M_3}.
\end{align*}
Then $f_3^{M_1,M_3}$ has modulation $M_3^4 \gg M^4$.

Observe that $  \xi_1-\xi_2+\xi_3 =\xi   $ with $\vert \xi \vert= M, \vert \xi_i \vert= M_3, i=1,2,3$. In order for the output to be at a low frequency $M$, two of the frequencies $\xi_1,-\xi_2,\xi_3$ should be $M_3$ separated. Therefore, we use the bilinear estimates for those two factors and the energy bound for the remainig factor to obtain
\begin{align*}
\Vert P_M\left(  v_{M_3} \overline{v_{M_1} } v_{M_3}  \right)  \Vert_{L_t^2L_x^2} \lesssim M_3^{-\frac{3}{2}}M^{\frac{1}{2}} \Vert v_{M_3} \Vert_{U_{S}^2} \Vert v_{M_3} \Vert_{U_{S}^2} \Vert v_{M_3} \Vert_{U_{S}^2}.
\end{align*} 
Hence, we obtain as in Case 1.a, Case 1.b
\begin{align*}
\sum\limits_{\substack{ J\subset I \\ \vert J \vert=M_3^{4s+2} }}\Vert f_3^{M_3} \Vert_{\dot{X}^{0,-\frac{1}{2},1}  } \lesssim M_3^{-7s-\frac{11}{2}}M^{4s+\frac{5}{2}}  \Vert u_{M_3}\Vert_{X^s}\Vert u_{M_3}\Vert_{X^s}\Vert u_{M_1}\Vert_{X^s},
\end{align*} 
which is summable with respect to $M_3$ if $s \geq -\frac{11}{14}$.
After summation with respect to $M_3$, we obtain 
\begin{align*}
\Vert f_3 \Vert_{\dot{X}^{0,-\frac{1}{2},1} }\lesssim M^{-3-3s}\Vert u \Vert_{X^s}^3.
\end{align*}

\textbf{Case 2.} In this case, there is at least one input which has large modulation. Depending on which factor has the large modulation and on whether the conjugated factor $\overline{u}_{N_2}$ has lower frequency or not, we divide this case into six.\\

\textbf{Subcase 2.a} We consider 
\begin{align*}
f_4
=&\sum\limits_{M\ll M_1 \ll M_3 } \sum\limits_{\substack{J\subset I \\ \vert J \vert=M_3^{4s+2} }} P_M \left( Q_{\gtrsim M_1M_3^3} \left(\chi_J u_{M_3} \right) \overline{\chi_J u_{M_3}}    \chi_J u_{M_1}  \right)\\
=& \sum\limits_{M\ll M_1 \ll M_3}  \sum\limits_{\substack{ J\subset I \\ \vert J \vert=M_3^{4s+2} }}f_4^{M_1,M_3},\\
f_5=& \sum\limits_{M\ll M_1\ll M_3}\sum\limits_{\substack{ J\subset I \\ \vert J \vert=M_3^{4s+2} }} P_M \left( \chi_J u_{M_3}   \overline{Q_{\gtrsim M_1M_3^3} \left(\chi_J u_{M_3}   \right)} \chi_J u_{M_1}  \right)
\end{align*}
Since the two terms $f_4$ and $f_5$ are similar, we may consider the first one $f_4$.

By using the embedding $V_A^2 \hookrightarrow \dot{X}^{0,\frac{1}{2},\infty}$ for the first factor and the bilinear estimates for the remainig terms, we have
\begin{align}\label{triple L^1 bound}
\Vert \left(Q_{\gtrsim M_1M_3^3}  v_{M_3}\right) \overline{v_{M_3}} v_{M_1} \Vert_{L_{t,x}^1}\lesssim M_1^{-\frac{1}{2}}M_3^{-3} \Vert v_{M_3} \Vert_{U^2_{S}} \Vert v_{M_3} \Vert_{U^2_{S}} \Vert v_{M_1} \Vert_{U^2_{S}}.
\end{align}

\textbf{Low modulation output.} By Bernstein's inequality, we have
\begin{align*}
\Vert P_M \left( \left(Q_{\gtrsim M_1M_3^3}  v_{M_3}\right) \overline{v_{M_3}} v_{M_1} \right)\Vert_{L_t^1 L_x^2} \lesssim M^{\frac{1}{2}} M_1^{-\frac{1}{2}}M_3^{-3} \Vert v_{M_3} \Vert_{U^2_{S}} \Vert u_{v_3} \Vert_{U^2_{S}} \Vert v_{M_1} \Vert_{U^2_{S}}.
\end{align*}
By considering the interval summation loss, we have
\begin{align*}
&\sum\limits_{\substack{ J\subset I \\ \vert J \vert=M_3^{4s+2} }}\Vert f_4^{M_1,M_3} \Vert_{L_t^1L_x^2}\\ 
&\lesssim M_1^{-s}M_3^{-2s} M^{4s+2}M_3^{-4s-2}M^{\frac{1}{2}}M_1^{-\frac{1}{2}}M_3^{-3}  \Vert u_{M_3}\Vert_{X^s}\Vert u_{M_3}\Vert_{X^s}\Vert u_{M_1}\Vert_{X^s}.
\end{align*}
This summation with respect to $M_1,M_3$ can be dealt with $s\geq -\frac{11}{14}$.
After summation with respect to $M_1, M_3$, we obtain 
\begin{align*}
\Vert Q_{\leq M^4} f_4 \Vert_{L_t^1L_x^2}
\lesssim& \Vert f_4 \Vert_{L_t^1L_x^2}\\
\lesssim& \sum\limits_{M\ll M_1 \ll M_3}\sum\limits_{\substack{ J\subset I \\ \vert J \vert=M_3^{4s+2} }} \Vert f_4^{M_1,M_3}\Vert_{L_t^1L_x^2} \\
\lesssim& M^{-3-3s}\Vert u \Vert_{X^s}^3.
\end{align*}

To estimate $ \Vert Q_{\geq M^4} f_4\Vert_{\dot{X}^{0,-\frac{1}{2},1}} $, we need to decompose $ Q_{\geq M^4} f_4$ into the following intermediate modulation output and high modulation output:
\begin{align*}
Q_{\geq M^4}f_4=&\sum\limits_{\substack{M_1,M_3:\\ M\ll M_1\ll M_3 }}\sum\limits_{\substack{ J\subset I \\ \vert J \vert=M_3^{4s+2} }}Q_{\geq M^4}f_4^{M_1,M_3}\\
=&\sum\limits_{\substack{M_1,M_3:\\ M\ll M_1\ll M_3 }} \sum\limits_{\substack{ J\subset I \\ \vert J \vert=M_3^{4s+2} }} Q_{M^4 \leq \sigma \lesssim M_1M_3^3}f_4^{M_1,M_3}+\sum\limits_{\substack{M_1,M_3:\\M\ll M_1 \ll M_3}} \sum\limits_{\substack{ J\subset I \\ \vert J \vert=M_3^{4s+2} }} Q_{\gg M_1M_3^3}f_4^{M_1,M_3}.
\end{align*}

\textbf{Intermediate modulation output.} In this case, we consider the $\dot{X}^{0,-\frac{1}{2},1}$ estimate at modulation $ M^4 \leq \sigma \lesssim M_1 M_3^3$.
By using $\eqref{triple L^1 bound}$ and Bernstein inequality we have
\begin{align*}
\Vert Q_{\sigma}P_M \left( Q_{ \gtrsim M_1M_3^3}v_{M_3}\overline{v_{M_3}} v_{M_1}  \right) \Vert_{L_{t,x}^2} \lesssim \left( M \sigma\right)^{\frac{1}{2}}M_1^{-\frac{1}{2}}M_3^{-3}  \Vert v_{M_3} \Vert_{U^2_{S}} \Vert v_{M_3} \Vert_{U^2_{S}} \Vert v_{M_1} \Vert_{U^2_{S}}.
\end{align*}
By considering the interval summation, we have
\begin{align*}
&\sum\limits_{\substack{ J\subset I \\ \vert J \vert=M_3^{4s+2} }}\Vert Q_\sigma  f_4^{M_1,M_3} \Vert_{L_{t,x}^2}\\
\lesssim& \sum\limits_{\substack{J\subset I \\ \vert J \vert=M_3^{4s+2}}} \Vert  Q_\sigma P_M \left( Q_{\gtrsim M_1M_3^3} \left(\chi_J u_{M_3} \right) \overline{\chi_J u_{M_3}}    \chi_J u_{M_1}  \right) \Vert_{L_{t,x}^2} \\
\lesssim& \; M^{4s+2}M_3^{-4s-2} \left( M \sigma\right)^{\frac{1}{2}}M_1^{-\frac{1}{2}}M_3^{-3}  \sup\limits_{\substack{ J\subset I \\ \vert J \vert=M_3^{4s+2} }} \Vert \chi_J u_{M_3} \Vert_{U^2_{S}} \Vert \chi_J u_{M_3} \Vert_{U^2_{S}} \Vert \chi_J u_{M_1} \Vert_{U^2_{S}}\\
\lesssim& \;\sigma^{\frac{1}{2}} M^{4s+\frac{5}{2}}M_1^{-\frac{1}{2}-s}M_3^{-6s-5} \Vert u_{M_3}\Vert_{X^s}\Vert u_{M_3}\Vert_{X^s}\Vert u_{M_1}\Vert_{X^s},
\end{align*}
or equivalently
\begin{align*}
\sum\limits_{\substack{ J\subset I \\ \vert J \vert=M_3^{4s+2} }}\Vert Q_\sigma  f_4^{M_1,M_3} \Vert_{\dot{X}^{0,-\frac{1}{2},1}} \lesssim & \; M^{4s+\frac{5}{2}}M_1^{-\frac{1}{2}-s}M_3^{-6s-5} \Vert u_{M_3}\Vert_{X^s}\Vert u_{M_3}\Vert_{X^s}\Vert u_{M_1}\Vert_{X^s}.
\end{align*}
By summing over $ M^4 \leq \sigma \lesssim M_1M_3^3 $, we obtain
\begin{align*}
&\sum\limits_{ M^4 \leq \sigma \lesssim M_1 M_3^3 }\sum\limits_{\substack{ J\subset I \\ \vert J \vert=M_3^{4s+2} }} \Vert Q_\sigma  f_4^{M_1,M_3} \Vert_{\dot{X}^{0,-\frac{1}{2},1}}\\
\lesssim & \; M^{4s+\frac{5}{2}}M_1^{-\frac{1}{2}-s}M_3^{-6s-5}\ln\left(\frac{M_3}{M}\right) \Vert u_{M_3}\Vert_{X^s}\Vert u_{M_3} \Vert_{X^s}\Vert u_{M_1}\Vert_{X^s}.
\end{align*}
Hence, the summation with respect to $M_1,M_3$ can be handled with $s>-\frac{11}{14}$.
After summation with respect to $M_1, M_3$, we obtain 
\begin{align}\label{intermediate modulation}
\sum\limits_{M \ll M_1 \ll M_3}\sum\limits_{\substack{ J\subset I \\ \vert J \vert=M_3^{4s+2} }} \Vert Q_{   M^4 \leq \sigma \lesssim  M_1 M_3^3} f_4^{M_1,M_3} \Vert_{\dot{X}^{0,-\frac{1}{2},1}}
\lesssim M^{-3-3s}\Vert u \Vert_{X^s}^3.
\end{align}

\textbf{High modulation output.} In this case, we need to estimate the output localized at modulations $\sigma \gg M_1 M_3^3$. In order to obtain such an output at least one of the inputs should have modulation at least $\sigma$. We assume that the lower frequency factor has modulation $\sigma$. This is the worst case in terms of bilinear separation. 

For the product $  v_{M_3} \overline{v_{M_3} } Q_\sigma v_{M_1}   $, we use the embedding $V_S^2 \hookrightarrow \dot{X}^{0,\frac{1}{2},1}$ for $Q_\sigma v_{M_1}$ and the bilinear estimate for $v_{M_3}\overline{v_{M_3}}$.
Observe that in order for the final output to be at frequency $M$, the two factors $v_{M_3} \overline{v_{M_3}}$ should be frequency localized in $M_1$ separated intervals of length $\vert \xi_1-\xi_2\vert \approx M_1$.
Therefore, by using the high modulation bound for $Q_\sigma v_{M_1}$ and the bilinear estimate for $v_{M_3}\overline{v_{M_3}}$ with bilinear gain $\left(M_1M_3^2\right)^{-\frac{1}{2}}$, we have
\begin{align*}
\Vert   v_{M_3} \overline{v_{M_3} } Q_\sigma v_{M_1}    \Vert_{L_t^1 L_x^1} \lesssim M_1^{-\frac{1}{2}}M_3^{-1} \sigma^{-\frac{1}{2}} \Vert v_{M_3} \Vert_{U_{S}^2} \Vert v_{M_3} \Vert_{U_{S}^2} \Vert v_{M_1} \Vert_{U_{S}^2}.
\end{align*} 
Applying $ Q_\sigma P_M  $ and the Bernstein's inequality, we obtain
\begin{align}
\Vert Q_\sigma P_M \left(  v_{M_3} \overline{v_{M_3} } Q_{\sigma}v_{M_1}  \right)  \Vert_{L_t^2L_x^2} \lesssim \left(\sigma M\right)^{\frac{1}{2}} M_1^{-\frac{1}{2}}M_3^{-1} \sigma^{-\frac{1}{2}} \Vert v_{M_3} \Vert_{U_{S}^2} \Vert v_{M_3} \Vert_{U_{S}^2} \Vert v_{M_1} \Vert_{U_{S}^2}.
\end{align}
By considering the interval summation, we have
\begin{align*}
&\sum\limits_{\substack{ J\subset I \\ \vert J \vert=M_3^{4s+2} }}\Vert Q_{\gg M_1M_3^3}  f_4^{M_1,M_3} \Vert_{\dot{X}^{0,-\frac{1}{2},1}}\\
\lesssim& \sum\limits_{\substack{J\subset I \\ \vert J \vert=M_3^{4s+2}}} \Vert  Q_{\gg M_1M_3^3} P_M \left( Q_{\gtrsim M_1M_3^3} \left(\chi_J u_{M_3} \right) \overline{\chi_J u_{M_3}}    \chi_J u_{M_1}  \right) \Vert_{\dot{X}^{0,-\frac{1}{2},1}}\\
\lesssim&  \sum\limits_{\substack{J\subset I \\ \vert J \vert=M_3^{4s+2}}}  \sum\limits_{\sigma \gtrsim M_1M_3^3 }  \Vert  Q_{\sigma } P_M \left( Q_{\gtrsim M_1M_3^3} \left(\chi_J u_{M_3} \right) \overline{\chi_J u_{M_3}}    \chi_J u_{M_1}  \right) \Vert_{\dot{X}^{0,-\frac{1}{2},1}}\\
\lesssim& \;M^{4s+2}M_3^{-4s-2} \left(M_1 M_3^3 \right)^{-\frac{1}{2}} \left(M^{\frac{1}{2}}M_1^{-\frac{1}{2}}M_3^{-1}\right) M_1^{-s}M_3^{-2s} \\
& \qquad\qquad\qquad \times \Vert u_{M_3}\Vert_{X^s} \Vert u_{M_3} \Vert_{X^s}\Vert u_{M_1}\Vert_{X^s}\\
=&M_3^{-6s-\frac{9}{2}}M_1^{-1-s}M^{4s+\frac{5}{2}} \Vert u_{M_3}\Vert_{X^s} \Vert u_{M_3} \Vert_{X^s}\Vert u_{M_1}\Vert_{X^s}.
\end{align*} 
Therefore, the summation with respect to $M_1,M_3$ is handled if $s \geq -\frac{3}{4}$.
After summation with respect to $M_1, M_3$, we obtain 
\begin{align}\label{high modulation}
\sum\limits_{M \ll M_1 \ll M_3} \sum\limits_{\substack{ J\subset I \\ \vert J \vert=M_3^{4s+2} }} \Vert Q_{   \gg M_1M_3^3} f_4^{M_1,M_3} \Vert_{\dot{X}^{0,-\frac{1}{2},1}}
\lesssim M^{-3-3s}\Vert u \Vert_{X^s}^3.
\end{align}
Therefore, by combining the intermediate modulation case $\eqref{intermediate modulation}$ and high modulation case $\eqref{high modulation}$, we have
\begin{align*}
\Vert Q_{\geq M^4} f_4 \Vert \lesssim M^{-3-3s} \Vert u \Vert_{X^s}^3.
\end{align*}
 
\textbf{Subcase 2.b.} In this case, the low frequency factor has high modulation. Here, we consider
\begin{align*}
f_6=&\sum\limits_{M\ll M_1 \ll M_3 } \sum\limits_{\substack{J\subset I \\ \vert J \vert=M_3^{4s+2} }} P_M \left(  \chi_J u_{M_3}  \overline{\chi_J u_{M_3}} Q_{\gtrsim M_1M_3^3} \left(\chi_J u_{M_1} \right)    \right)\\
=& \sum\limits_{M\ll M_1 \ll M_3}\sum\limits_{\substack{ J\subset I \\ \vert J \vert=M_3^{4s+2} }} f_6^{M_1,M_3}.
\end{align*} 
For the $  v_{M_3} \overline{v_{M_3} } Q_{\gtrsim M_1M_3^3} v_{M_1}   $, we use the embedding $V_S^2 \hookrightarrow \dot{X}^{0,\frac{1}{2},1}$ for $Q_{\gtrsim M_1M_3^3} v_{M_1}$ and the bilinear estimate for $v_{M_3}\overline{v_{M_3}}$.
Observe that in order for the final output to be at frequency $M$, the two factors $v_{M_3} \overline{v_{M_3}}$ should be frequency localized in $M_1$ separated intervals of length $\vert \xi_1-\xi_2\vert \approx M_1$.
Therefore, by using the high modulation bound for $Q_{\gtrsim M_1M_3^3} v_{M_1}$ and the bilinear estimate for $v_{M_3}\overline{v_{M_3}}$ with bilinear gain $\left(M_1M_3^2\right)^{-\frac{1}{2}}$, we have
\begin{align*}
\Vert   v_{M_3} \overline{v_{M_3} } Q_{\gtrsim M_1M_3^3} v_{M_1}    \Vert_{L_t^1 L_x^1} \lesssim & M_1^{-\frac{1}{2}}M_3^{-1} \left(M_1M_3^3\right)^{-\frac{1}{2}} \Vert v_{M_3} \Vert_{U_{S}^2} \Vert v_{M_3} \Vert_{U_{S}^2} \Vert v_{M_1} \Vert_{U_{S}^2}\\
=&M_1^{-1}M_3^{-\frac{5}{2}} \Vert v_{M_3} \Vert_{U_{S}^2} \Vert v_{M_3} \Vert_{U_{S}^2} \Vert v_{M_1} \Vert_{U_{S}^2}.
\end{align*} 

\textbf{Low modulation output.} By Bernstein's inequality, we have
\begin{align}\label{low frequency high modulation}
\Vert P_M \left( \overline{v_{M_3}} v_{M_3}  \left(Q_{\gtrsim M_1M_3^3}  v_{M_1}\right)  \right)\Vert_{L_t^1 L_x^2} \lesssim M^{\frac{1}{2}} M_1^{-1}M_3^{-\frac{5}{2}} \Vert v_{M_3} \Vert_{U^2_{S}} \Vert v_{M_3} \Vert_{U^2_{S}} \Vert v_{M_1} \Vert_{U^2_{S}}.
\end{align}
By considering the interval summation loss, we have
\begin{align*}
&\sum\limits_{\substack{ J\subset I \\ \vert J \vert=M_3^{4s+2} }}\Vert f_6^{M_1,M_3} \Vert_{L_t^1L_x^2}\\ 
\lesssim & M_1^{-s}M_3^{-2s} M^{4s+2}M_3^{-4s-2}M^{\frac{1}{2}}M_1^{-1}M_3^{-\frac{5}{2}}  \Vert u_{M_3}\Vert_{X^s}\Vert u_{M_3}\Vert_{X^s}\Vert u_{M_1}\Vert_{X^s}\\
=& M^{4s+\frac{5}{2}}M_3^{-6s-\frac{9}{2}}M_1^{-s-1}  \Vert u_{M_3}\Vert_{X^s}\Vert u_{M_3}\Vert_{X^s}\Vert u_{M_1}\Vert_{X^s}.
\end{align*}
This summation with respect to $M_1,M_3$ can be dealt with $s\geq -\frac{3}{4}$.
After summation with respect to $M_1, M_3$, we obtain 
\begin{align*}
\Vert Q_{\leq M^4} f_6 \Vert_{L_t^1L_x^2}
\lesssim& \Vert f_6 \Vert_{L_t^1L_x^2}\\
\lesssim& \sum\limits_{M\ll M_1 \ll M_3}\sum\limits_{\substack{ J\subset I \\ \vert J \vert=M_3^{4s+2} }} \Vert f_6^{M_1,M_3}\Vert_{L_t^1L_x^2} \\
\lesssim& M^{-3-3s}\Vert u \Vert_{X^s}^3.
\end{align*}

To estimate $ \Vert Q_{\geq M^4} f_6\Vert_{\dot{X}^{0,-\frac{1}{2},1}} $, we need to decompose $ Q_{\geq M^4} f_6$ into the following intermediate modulation output and high modulation output:
\begin{align*}
Q_{\geq M^4}f_6=&\sum\limits_{\substack{M_1,M_3:\\ M\ll M_1\ll M_3 }}\sum\limits_{\substack{ J\subset I \\ \vert J \vert=M_3^{4s+2} }} Q_{\geq M^4}f_6^{M_1,M_3}\\
=&\sum\limits_{\substack{M_1,M_3:\\ M\ll M_1\ll M_3 }}\sum\limits_{\substack{ J\subset I \\ \vert J \vert=M_3^{4s+2} }} Q_{M^4 \leq \sigma \lesssim M_1M_3^3}f_6^{M_1,M_3}+\sum\limits_{\substack{M_1,M_3:\\M\ll M_1 \ll M_3}}\sum\limits_{\substack{ J\subset I \\ \vert J \vert=M_3^{4s+2} }}Q_{\gg M_1M_3^3}f_6^{M_1,M_3}.
\end{align*}

\textbf{Intermediate modulation output.} We consider the $\dot{X}^{0,-\frac{1}{2},1}$ estimate at modulation $ M^4 \leq \sigma \lesssim M_1 M_3^3$.
By using $\eqref{low frequency high modulation}$ and Bernstein's inequality we have
\begin{align*}
\Vert Q_{\sigma}P_M \left(\left( \overline{v_{M_3}} v_{M_3}  \right) Q_{ \gtrsim M_1M_3^3}v_{M_1}\right) \Vert_{L_{t,x}^2} \lesssim \left( M \sigma\right)^{\frac{1}{2}}M_1^{-1}M_3^{-\frac{5}{2}}  \Vert v_{M_3} \Vert_{U^2_{S}} \Vert v_{M_3} \Vert_{U^2_{S}} \Vert v_{M_1} \Vert_{U^2_{S}}.
\end{align*}
By considering the interval summation, we have
\begin{align*}
\sum\limits_{\substack{ J\subset I \\ \vert J \vert=M_3^{4s+2} }}\Vert Q_\sigma  f_6^{M_1,M_3} \Vert_{L_{t,x}^2}\lesssim& \sum\limits_{\substack{J\subset I \\ \vert J \vert=M_3^{4s+2}}} \Vert   Q_{\sigma}P_M \left(\left( \chi_J\overline{u_{M_3}} \chi_Ju_{M_3}  \right) Q_{ \gtrsim M_1M_3^3}\chi_J u_{M_1}\right) \Vert_{L_{t,x}^2} \\
\lesssim& \; M^{4s+2}M_3^{-4s-2} \left( M \sigma\right)^{\frac{1}{2}}M_1^{-1}M_3^{-\frac{5}{2}}\\
 \times&  \sup\limits_{\substack{ J\subset I \\ \vert J \vert=M_3^{4s+2} }} \Vert \chi_J u_{M_3} \Vert_{U^2_{S}} \Vert \chi_J u_{M_3} \Vert_{U^2_{S}} \Vert \chi_J u_{M_1} \Vert_{U^2_{S}}\\
\lesssim& \;\sigma^{\frac{1}{2}} M^{4s+\frac{5}{2}}M_1^{-1-s}M_3^{-6s-\frac{9}{2}} \Vert u_{M_3}\Vert_{X^s}\Vert u_{M_3}\Vert_{X^s}\Vert u_{M_1}\Vert_{X^s},
\intertext{or equivalently}
\sum\limits_{\substack{ J\subset I \\ \vert J \vert=M_3^{4s+2} }}\Vert Q_\sigma  f_6^{M_1,M_3} \Vert_{\dot{X}^{0,-\frac{1}{2},1}} \lesssim & \; M^{4s+\frac{5}{2}}M_1^{-1-s}M_3^{-6s-\frac{9}{2}} \Vert u_{M_3}\Vert_{X^s}\Vert u_{M_3}\Vert_{X^s}\Vert u_{M_1}\Vert_{X^s}.
\end{align*}
By summing over $ M^4 \leq \sigma \lesssim M_1M_3^3 $, we obtain
\begin{align*}
\sum\limits_{\substack{ J\subset I \\ \vert J \vert=M_3^{4s+2} }}\sum\limits_{ M^4 \leq \sigma \lesssim M_1M_3^3 }\Vert Q_\sigma  f_6^{M_1,M_3} \Vert_{\dot{X}^{0,-\frac{1}{2},1}} \lesssim & \; M^{4s+\frac{5}{2}}M_1^{-1-s}M_3^{-6s-\frac{9}{2}}\ln \left(\frac{M_3}{M}\right)\\
\times& \Vert u_{M_3}\Vert_{X^s}\Vert u_{M_3} \Vert_{X^s}\Vert u_{M_1}\Vert_{X^s}.
\end{align*}
Hence, the summation with respect to $M_1,M_3$ can be handled with $s>-\frac{3}{4}$.
After summation with respect to $M_1, M_3$, we obtain 
\begin{align}\label{intermediate modulation 2}
\sum\limits_{M \ll M_1 \ll M_3}\sum\limits_{\substack{ J\subset I \\ \vert J \vert=M_3^{4s+2} }} \Vert Q_{   M^4 \leq \sigma \lesssim M_1M_3^3} f_6^{M_1,M_3} \Vert_{\dot{X}^{0,-\frac{1}{2},1}}
\lesssim M^{-3-3s}\Vert u \Vert_{X^s}^3.
\end{align}

\textbf{High modulation output.} In this case, we need to estimate the output localized at modulations $\sigma \gg M_1M_3^3$. In order to obtain such an output at least one of the inputs should have modulation at least $\sigma$. We assume that the lower frequency factor has modulation $\sigma$. This is the worst case in terms of bilinear separation. 

For the product $  v_{M_3} \overline{v_{M_3} } Q_\sigma v_{M_1}   $, we use the embedding $V_S^2 \hookrightarrow \dot{X}^{0,\frac{1}{2},1}$ for $Q_\sigma v_{M_1}$ and the bilinear estimate for $v_{M_3}\overline{v_{M_3}}$.
Observe that in order for the final output to be at frequency $M$, the two factors $v_{M_3} \overline{v_{M_3}}$ should be frequency localized in $M_1$ separated intervals of length $\vert \xi_1-\xi_2\vert \approx M_1$.
Therefore, by using the high modulation bound for $Q_\sigma v_{M_1}$ and the bilinear estimate for $v_{M_3}\overline{v_{M_3}}$ with bilinear gain $\left(M_1M_3^2\right)^{-\frac{1}{2}}$, we have
\begin{align*}
\Vert   v_{M_3} \overline{v_{M_3} } Q_\sigma v_{M_1}    \Vert_{L_t^1 L_x^1} \lesssim M_1^{-\frac{1}{2}}M_3^{-1} \sigma^{-\frac{1}{2}} \Vert v_{M_3} \Vert_{U_{S}^2} \Vert v_{M_3} \Vert_{U_{S}^2} \Vert v_{M_1} \Vert_{U_{S}^2}.
\end{align*} 
Applying $ Q_\sigma P_M  $, Bernstein's inequality, 
and considering the interval summation, we have
\begin{align*}
\sum\limits_{\substack{ J\subset I \\ \vert J \vert=M_3^{4s+2} }}\Vert Q_{\gg M_1M_3^3}  f_6^{M_1,M_3} \Vert_{L_{t,x}^2}\lesssim& \sum\limits_{\substack{J\subset I \\ \vert J \vert=M_3^{4s+2}}} \Vert  Q_{\gg M_1M_3^3} P_M \left( \chi_J u_{M_3}  \overline{\chi_J u_{M_3}} Q_{\gtrsim M_1M_3^3} \left(\chi_J u_{M_1} \right)  \right) \Vert_{L_{t,x}^2}\\
\lesssim&  \sum\limits_{\substack{J\subset I \\ \vert J \vert=M_3^{4s+2}}}  \sum\limits_{\sigma \gtrsim M_1M_3^3 }  \Vert  Q_{\sigma } P_M \left(\chi_J u_{M_3}  \overline{\chi_J u_{M_3}} Q_{\gtrsim M_1M_3^3} \left(\chi_J u_{M_1} \right)  \right) \Vert_{L_{t,x}^2},
\end{align*}
or equivalently
\begin{align*}
\sum\limits_{\substack{ J\subset I \\ \vert J \vert=M_3^{4s+2} }}\Vert Q_{\gg M_1M_3^3}  f_6^{M_1,M_3} \Vert_{\dot{X}^{0,-\frac{1}{2},1}}\lesssim& \;M^{4s+2}M_3^{-4s-2} \left(M_1 M_3^3 \right)^{-\frac{1}{2}} \left(M^{\frac{1}{2}}M_1^{-\frac{1}{2}}M_3^{-1}\right) M_1^{-s}M_3^{-2s} \\
& \qquad\qquad\qquad \times \Vert u_{M_3}\Vert_{X^s} \Vert u_{M_3} \Vert_{X^s}\Vert u_{M_1}\Vert_{X^s}\\
=&M_3^{-6s-\frac{9}{2}}M_1^{-1-s}M^{4s+\frac{5}{2}} \Vert u_{M_3}\Vert_{X^s} \Vert u_{M_3} \Vert_{X^s}\Vert u_{M_1}\Vert_{X^s}.
\end{align*}
Therefore, the summation with respect to $M_1,M_3$ is handled if $s \geq -\frac{3}{4}$.
After summation with respect to $M_1, M_3$, we obtain 
\begin{align}\label{higm modulation 2}
\sum\limits_{M \ll M_1 \ll M_3} \sum\limits_{\substack{ J\subset I \\ \vert J \vert=M_3^{4s+2} }}  \Vert  Q_{   \gg M_1M_3^3} f_6^{M_1,M_3} \Vert_{\dot{X}^{0,-\frac{1}{2},1}}
\lesssim M^{-3-3s}\Vert u \Vert_{X^s}^3.
\end{align} 
Therefore, by combining the intermediate modulation case $\eqref{intermediate modulation 2}$ and high modulation case $\eqref{higm modulation 2}$, we have
\begin{align*}
\Vert Q_{\geq M^4} f_6 \Vert \lesssim M^{-3-3s} \Vert u \Vert_{X^s}^3.
\end{align*}

\textbf{Subcae 2.c} In this case, the low frequency factor is conjugated but does not have high modulation. We consider
\begin{align*}
f_7
=&\sum\limits_{M\ll M_1 \lesssim M_3 } \sum\limits_{\substack{J\subset I \\ \vert J \vert=M_3^{4s+2} }} P_M \left(  Q_{\gtrsim M_3^4} \left(\chi_J u_{M_3} \right) \chi_J u_{M_3}  \overline{\chi_J u_{M_1}}    \right)\\
=& \sum\limits_{M\ll M_1 \lesssim M_3}\sum\limits_{\substack{ J\subset I \\ \vert J \vert=M_3^{4s+2} }} f_7^{M_1,M_3}
\end{align*}
If $M_1\ll M_3$, then the last two factor are $M_3$ separated in frequency. Although if $M_1 \approx M_3$, in order for the resulting frequency to be localized at frequency $M$, the two last factor should be still $ \vert \xi_3-\xi_2 \vert\approx M_3$ separated.    
Therefore, by using the bilinear estimate and high modulation bound, we can obtain the trilinear estimate:
\begin{align*}
\Vert \left(Q_{\gtrsim M_1M_3^3}  v_{M_3}\right) \overline{v_{M_1}} v_{M_3} \Vert_{L_{t,x}^1}\lesssim M_1^{-\frac{1}{2}}M_3^{-3} \Vert v_{M_3} \Vert_{U^2_{S}} \Vert v_{M_3} \Vert_{U^2_{S}} \Vert v_{M_1} \Vert_{U^2_{S}}.
\end{align*}
Therefore, the rest of the argument proceeds as in Subcase 2.a without any significant changes.

\textbf{Subcae 2.d} In this case, the low frequency factor is conjugated and has high modulation. We consider 
\begin{align*}
f_8
=&\sum\limits_{M\ll M_1 \ll M_3 } \sum\limits_{\substack{J\subset I \\ \vert J \vert=M_3^{4s+2} }} P_M \left( \chi_J u_{M_3}  \overline{Q_{\gtrsim M_3^4} \left(\chi_J u_{M_1} \right)} \chi_J u_{M_3}     \right)\\
=& \sum\limits_{M\ll M_1 \ll M_3}\sum\limits_{\substack{ J\subset I \\ \vert J \vert=M_3^{4s+2} }} f_8^{M_1,M_3}.
\end{align*}
In order for the resulting frequency to be localized at frequency $M$, the two frequency $M_3$ factors should be still $  M_3$ separated. Therefore, by applying the bilinear estimate and high modulation bound, we obtain the trilinear estimate
\begin{align*}
\Vert  v_{M_3}  \overline{Q_{\gtrsim M_3^4} v_{M_1}  } v_{M_3}  \Vert_{L_{t,x}^1}\lesssim & M_3^{-\frac{7}{2}}  \Vert v_{M_3} \Vert_{U^2_{S}} \Vert v_{M_3} \Vert_{U^2_{S}} \Vert v_{M_1} \Vert_{U^2_{S}}.
\end{align*}
Therefore, we can argue as in Subcase 2.a with better gains.

\textbf{Subcase 2.e} In this case, all frequencies are equal and the conjugated factor has high modulation. We consider
\begin{align*} 
f_9=&\sum\limits_{M\ll M_3 } \sum\limits_{\substack{J\subset I \\ \vert J \vert=M_3^{4s+2} }} P_M \left( \chi_J u_{M_3}  \overline{Q_{\gtrsim M_3^4} \left(\chi_J u_{M_3} \right)} \chi_J u_{M_3}     \right)\\
=& \sum\limits_{M\ll M_3}\sum\limits_{\substack{ J\subset I \\ \vert J \vert=M_3^{4s+2} }}f_9^{M_3}.
\end{align*}
This is the worst case since we cannot have any frequency separation among the two unconjugated factors and hence we cannot depend on bilinear gains. Therefore, to obtain summability for $M_3$, we use the local smoothing estimates.
By using the local smoothing estimates $\eqref{eqn:local smoothing estimate U^2}$, maximal function estimates $\eqref{eqn:maximal function estimate U^4}$ and high modulation bound, we have
\begin{equation}\label{local smoothing + maximal function}
\begin{split}
\Vert v_{M_3} \overline{Q_{\gtrsim M_3^4}v_{M_3}} v_{M_3} \Vert_{L_x^{\frac{4}{3}}L_t^1  }\lesssim& \Vert v_{M_3} \Vert_{L_x^{\infty}L_t^2} \Vert Q_{\gtrsim M_3^4 }v_{M_3} \Vert_{L_x^2 L_t^2} \Vert v_{M_3} \Vert_{L_x^4L_t^{\infty}}  \\
\lesssim & M_3^{-\frac{13}{4}}  \Vert v_{M_3} \Vert_{U^2_{S}} \Vert v_{M_3} \Vert_{U^2_{S}} \Vert v_{M_3} \Vert_{U^2_{S}}.
\end{split}
\end{equation}

\textbf{Low modulation output}
By considering the interval summation, we obtain
\begin{align*}
\sum\limits_{\substack{ J\subset I \\ \vert J \vert=M_3^{4s+2} }}\Vert f_9^{M_3} \Vert_{L_x^{\frac{4}{3}}L_t^1 }\lesssim & \left( M^{4s+2} M_3^{-4s-2} \right) M_3^{-3s}M_3^{-\frac{13}{4}} \Vert u_{M_3} \Vert_{X^s}^3\\
\lesssim& M^{4s+2}M_3^{-7s-\frac{21}{4}} \Vert u_{M_3} \Vert_{X^s}^3,
\end{align*}
which is summable with respect to $M_3$ if $s\geq -\frac{3}{4}$.
After summation with respect to $ M_3$, we obtain 
\begin{align*}
\Vert Q_{\leq M^4} f_9 \Vert_{L_x^{\frac{4}{3}}L_t^1}
\lesssim  \Vert f_9 \Vert_{L_x^{\frac{4}{3}}L_t^1}
\lesssim M^{-3s-\frac{13}{4}}\Vert u \Vert_{X^s}^3.
\end{align*}

To estimate $ \Vert Q_{\geq M^4} f_9\Vert_{\dot{X}^{0,-\frac{1}{2},1}} $, we need to decompose $ Q_{\geq M^4} f_9$ into the following intermediate modulation output and high modulation output:
\begin{align*}
Q_{\geq M^4}f_9=&\sum\limits_{\substack{M_3:\\ M\ll M_3 }}\sum\limits_{\substack{ J\subset I \\ \vert J \vert=M_3^{4s+2} }}Q_{\geq M^4}f_9^{M_3}\\
=&\sum\limits_{\substack{M_3:\\ M\ll  M_3 }}\sum\limits_{\substack{ J\subset I \\ \vert J \vert=M_3^{4s+2} }} Q_{M^4 \leq \sigma \lesssim M_3^4}f_9^{M_3}+\sum\limits_{\substack{M_3:\\M\ll  M_3}}\sum\limits_{\substack{ J\subset I \\ \vert J \vert=M_3^{4s+2} }} Q_{\gg M_3^4}f_9^{M_3}.
\end{align*} 

\textbf{Intermediate modulation output}
By using Bernstein's inequality and $\eqref{local smoothing + maximal function}$, we have
\begin{align*}
\Vert Q_\sigma P_M \left( v_{M_3} \overline{Q_{\gtrsim M_3^4}v_{M_3}} v_{M_3} \right) \Vert_{L_x^{2} L_t^2  }\lesssim \sigma^{\frac{1}{2}}M^{\frac{1}{4}} M_3^{-\frac{13}{4}} \Vert v_{M_3} \Vert_{U^2_{S}} \Vert v_{M_3} \Vert_{U^2_{S}} \Vert v_{M_3} \Vert_{U^2_{S}}.
\end{align*}
Hence, by considering the interval summation, we obtain
\begin{align*}
\sum\limits_{\substack{ J\subset I \\ \vert J \vert=M_3^{4s+2} }}\Vert Q_\sigma f_9^{M_3} \Vert_{L_{t,x}^2}\lesssim& \left( M^{4s+2} M_3^{-4s-2}\right)\sigma^{\frac{1}{2}}M^{\frac{1}{4}}M_3^{-\frac{13}{4}}M_3^{-3s}  \Vert u_{M_3} \Vert_{X^s}^3\\
=& M^{4s+\frac{9}{4}}M_3^{-7s-\frac{21}{4}}\sigma^{\frac{1}{2}}  \Vert u_{M_3} \Vert_{X^s}^3,
\intertext{or equivalently}
\sum\limits_{\substack{ J\subset I \\ \vert J \vert=M_3^{4s+2} }}\Vert Q_\sigma f_9^{M_3} \Vert_{\dot{X}^{0,-\frac{1}{2},1}}\lesssim & M^{4s+\frac{9}{4}}M_3^{-7s-\frac{21}{4}}  \Vert u_{M_3} \Vert_{X^s}^3.
\end{align*}  
By summing over $ M^4 \leq \sigma \leq M_3^4 $, we obtain
\begin{align*}
\sum\limits_{ M^4 \leq \sigma \lesssim M_3^4 } \sum\limits_{\substack{ J\subset I \\ \vert J \vert=M_3^{4s+2} }} \Vert Q_\sigma  f_9^{M_3} \Vert_{\dot{X}^{0,-\frac{1}{2},1}} \lesssim & \; M^{4s+\frac{9}{4}}M_3^{-7s-\frac{21}{4}}\ln\left(\frac{M_3}{M}\right)  \Vert u_{M_3}\Vert_{X^s}^3.
\end{align*}
Hence, the summation with respect to $M_3$ can be handled with $s>-\frac{3}{4}$.
After summation with respect to $ M_3$, we obtain 
\begin{align}\label{intermediate modulation 3}
\sum\limits_{M \ll  M_3}\sum\limits_{\substack{ J\subset I \\ \vert J \vert=M_3^{4s+2} }}\Vert  Q_{   M^4 \leq \sigma \lesssim M_3^4} f_9^{M_3} \Vert_{\dot{X}^{0,-\frac{1}{2},1}}
\lesssim M^{-3-3s}\Vert u \Vert_{X^s}^3.
\end{align} 

\textbf{High modulation output.} 
In this case, we need to estimate the output localized at modulations $\sigma \gg M_3^4$. In order to obtain such an output at least one of the inputs should have modulation at least $\sigma$. We assume that the conjugated factor has modulation $\sigma$. This is the worst case since, for the remaining case, we can use the bilinear estimates.
By using Bernstein's inequality, we have
\begin{align*}
\Vert Q_\sigma P_M \left(  v_{M_3} \overline{Q_\sigma v_{M_3} } v_{M_3}  \right) \Vert_{L_t^2L_x^2}\lesssim \sigma^{\frac{1}{2}}M^{\frac{1}{4}} \Vert  v_{M_3} \overline{Q_\sigma v_{M_3} } v_{M_3} \Vert_{L_x^{\frac{4}{3}}L_t^1}.
\end{align*}
By applying local smoothing estimate, maximal function estimate and high modulation bound, we have
\begin{align*}
\Vert Q_\sigma P_M \left(  v_{M_3} \overline{Q_\sigma v_{M_3} } v_{M_3}  \right) \Vert_{L_t^2L_x^2}\lesssim& \sigma^{\frac{1}{2}}M^{\frac{1}{4}}M_3^{-\frac{3}{2}}M_3^{\frac{1}{4}}\sigma^{-\frac{1}{2}} \Vert v_{M_3} \Vert_{U_S^2}\\
=&M^{\frac{1}{4}}M_3^{-\frac{5}{4}} \Vert v_{M_3} \Vert_{U_S^2}^3. 
\end{align*} 
By considering the interval summation, we obtain
\begin{align*}
\sum\limits_{\substack{ J\subset I \\ \vert J \vert=M_3^{4s+2} }}\Vert Q_{\gg M_3^4} f_9^{M_3} \Vert_{\dot{X}^{0,-\frac{1}{2},1}}\lesssim&  \sum\limits_{\substack{ J\subset I \\ \vert J \vert=M_3^{4s+2} }}\sum\limits_{\sigma \gtrsim M_3^4} \Vert Q_{\sigma} f_9^{M_3} \Vert_{\dot{X}^{0,-\frac{1}{2},1}}\\
\lesssim&\left(M^{4s+2}M_3^{-4s-2}\right) M^{\frac{1}{4}}M_3^{-\frac{13}{4}-3s} \Vert u_{M_3}\Vert_{X^s}^3\\
=&M^{4s+\frac{9}{4}}M_3^{-7s-\frac{21}{4}} \Vert u_{M_3}\Vert_{X^s}^3,
\end{align*}
which is summable with respect to $M_3$ if $ s \geq -\frac{3}{4}$. After summation with respect to $M_3$, we obtain 
\begin{align}\label{higm modulation 3}
\sum\limits_{M \ll  M_3}\sum\limits_{\substack{ J\subset I \\ \vert J \vert=M_3^{4s+2} }} \Vert Q_{   \gg M_3^4} f_9^{M_3} \Vert_{\dot{X}^{0,-\frac{1}{2},1}}
\lesssim M^{-3-3s}\Vert u \Vert_{X^s}^3.
\end{align} 
Therefore, by combining the intermediate modulation case $\eqref{intermediate modulation 3}$ and high modulation case $\eqref{higm modulation 3}$, we have
\begin{align*}
\Vert Q_{\geq M^4} f_9 \Vert \lesssim M^{-3-3s} \Vert u \Vert_{X^s}^3.
\end{align*}  
 
\end{proof}

\section{Conservation of the $H^s$ energy}\label{sec:energy estimate}
In this section, we want to show the conservation of the $H^s$ energy. 
We are inspired by the method that is analogous to that in Colliander-Keel-Staffilani-Takaoka-Tao \cite{CKSTT2003} and follow the argument in Koch-Tataru \cite{KT2007}. To obtain the energy estimate, we use the $I$-method with correction term.

We first define the $H^s$ energy: 
\begin{align*}
E_0\left(u\right)=\left\langle  a\left(D\right)u,u \right\rangle.
\end{align*}
For the $H^s$ energy conservation, we want to choose the symbol $a\left(\xi\right)=\left(1+\xi^2\right)^s$, but as in \cite{KT2007}, we allow a slightly larger class of symbols.
\begin{definition} Let $\epsilon>0 ,s \in \mathbb{R}$. Then $S_\epsilon^s$ is the class of spherically symmetric symbols with the following properties:
	
(i) Slowly varying condition: For $\vert \xi \vert \approx \vert \xi'\vert$, we have
\begin{align*}
a\left(\xi\right) \approx a\left(\xi'\right).
\end{align*}

(ii) symbol regularity,
\begin{align*}
\left\vert \partial^{\alpha} a\left(\xi\right) \right\vert \lesssim a\left(\xi\right)\left(1+\xi^2\right)^{-\frac{\alpha}{2}}.
\end{align*}	

(iii) decay at infinity,
\begin{align*}
s-\epsilon \leq \frac{\log a\left(\xi\right)}{\log\left(1+\xi^2 \right)}\leq s+\epsilon.
\end{align*}
Here $\epsilon$ is a small parameter.  
\end{definition}

The main goal of this section is to obtain the following energy bound.
\begin{prop}\label{prop:energy bound.}
Let $-\frac{3}{4} < s < -\frac{1}{2} $ and $u$ be a solution  to \eqref{eqn:fourth order NLS} with
\begin{align*}
\Vert u \Vert_{\ell_N^2 L_t^\infty H^s}\ll 1.
\end{align*}
Then we have
\begin{align}\label{eqn:energy estimates N.}
\Vert u \Vert_{\ell_N^2 L_t^\infty H^s} \lesssim \Vert u_0 \Vert_{H^s}+\Vert u \Vert_{X^s}^3.
\end{align}
\end{prop}

We define the energy functional
\begin{align*}
E_0\left(u\right)=\Vert u \Vert_{H^a}^2=\langle a(D)u,u\rangle_{L_x^2}.
\end{align*}
By differentiating this energy under the $\eqref{eqn:fourth order NLS}$ flow, we have
\begin{align*}
\frac{d}{dt}E_0\left(u\right)=2\Re \int_{\xi_1-\xi_2+\xi_3-\xi_4=0}i a\left(\xi_1\right)\widehat{u}\left(\xi_1\right)\overline{\widehat{u}}\left(\xi_2\right)\widehat{u}\left(\xi_3\right)\overline{\widehat{u}}\left(\xi_4\right).
\end{align*}
By symmetrizing above integral, we have
\begin{align*}
\frac{d}{dt}E_0\left(u\right)=&\frac{1}{2} \Re \int_{\xi_1-\xi_2+\xi_3-\xi_4=0}i\left(a\left(\xi_1\right)-a\left(\xi_2\right)+a\left(\xi_3\right)-a\left(\xi_4\right) \right)\widehat{u}\left(\xi_1\right)\overline{\widehat{u}}\left(\xi_2\right)\widehat{u}\left(\xi_3\right)\overline{\widehat{u}}\left(\xi_4\right)
\end{align*}
We want to cancel this term by adding the correction term $E_1\left(u\right)$, where $E_1\left(u\right)$ has the form
\begin{align*}
E_{1}\left(u\right)=\int\limits_{\xi_1-\xi_2+\xi_3-\xi_4=0}b_4\left(\xi_1,\dots,\xi_4\right) \widehat{u}\left(\xi_1\right)\overline{\widehat{u}}\left(\xi_2\right)\widehat{u}\left(\xi_3\right)\overline{\widehat{u}}\left(\xi_4\right),
\end{align*}
where the function $b_4$ is symmetric under the even $\xi_j$ indices, or of the odd $\xi_j$ indices.
The $b_4$ will be determined later. The role of $b_4$ is to make a cancelation. 
Observe that
\begin{align*}
\frac{d}{dt}E_1\left(u\right)=&\int\limits_{\xi_1-\xi_2+\xi_3-\xi_4=0}ib_4\left(\xi_1,\dots,\xi_4\right)\left(\xi_1^4-\xi_2^4+\xi_3^4-\xi_4^4 \right) \widehat{u}\left(\xi_1\right)\overline{\widehat{u}}\left(\xi_2\right)\widehat{u}\left(\xi_3\right)\overline{\widehat{u}}\left(\xi_4\right) \\
+&4\Re \int\limits_{\xi_1-\xi_2+\xi_3-\xi_4=0} ib_4\left(\xi_1,\xi_2,\xi_3,\xi_4 \right) \widehat{ u}\left(\xi_1\right)\overline{\widehat{u}}\left(\xi_2\right)\widehat{u}\left(\xi_3\right)\overline{\widehat{\vert u \vert^2 u}}\left(\xi_4\right).
\end{align*}
To cancel the first integral in $\frac{d}{dt}E_1\left(u\right)$, we choose $b_4$ as follows:
\begin{align*}
b_4\left(\xi_1,\dots,\xi_4 \right)=-\frac{a\left(\xi_1\right)-a\left(\xi_2\right)+a\left(\xi_3\right)-a\left(\xi_4\right)}{i\left(\xi_1^4-\xi_2^4+\xi_3^4 -\xi_4^4\right)}, \quad \text{on} \quad  P_4=\left\{\xi_1-\xi_2+\xi_3-\xi_4=0 \right\} .
\end{align*}
In this situation, a resonant interaction does not appear. Later in Proposition \ref{lem:multiplier estimate b_4}, we will show that the multiplier $b_4$ should be fully nonresonant.

Therefore by using the above calculations, we have
\begin{equation}
\label{eqn:Lambda_6}
\begin{split}
\Lambda_6\left(u(t)\right):=&\frac{d}{dt}\left(E_0+E_1 \right)\left(u\right)\\
=& 4\Re \int\limits_{\substack{\xi_1-\xi_2+\xi_3=\xi\\ \xi_4-\xi_5+\xi_6=\xi}} ib_4\left(\xi_1,\xi_2,\xi_3,\xi \right) \widehat{ u}\left(\xi_1\right)\overline{\widehat{u}}\left(\xi_2\right)\widehat{u}\left(\xi_3\right) \overline{\widehat{u}}\left(\xi_4\right) \widehat{u}\left(\xi_5\right) \overline{\widehat{u}}\left(\xi_6\right).
\end{split}
\end{equation} 
Before we prove the energy estimate, we prove the following multiplier estimate $\eqref{eqn:multiplier estimate in energy estimate}$ that shows the multiplier $b_4$ is fully nonresonant. In order to estimate the correction term $E_1\left(u\right)$ and the derivative of modified energy $\frac{d}{dt}\left(E_0+E_1\right)$, we need to obtain the size of $b_4$. 
Originally, $b_4$ is defined only on the diagonal $\left\{\xi_1-\xi_2+\xi_3-\xi_4=0 \right\}$. In order to separate variables, we want to extend it off diagonal in a smooth way.
  
Before stating the Lemma \ref{lem:multiplier estimate b_4}, we recall the following two mean value formulas: if $|\eta|,|\lambda|\ll |\xi|$, then
\begin{align}\label{eqn: mean value theorem}
\left| a(\xi+\eta)-a(\xi) \right|\lesssim |\eta|\sup\limits_{|
	\xi'|\approx |\xi|}\left|a'\left(\xi'\right) \right|,    
\end{align}
and
\begin{align}\label{eqn: double mean value theorem}
\left| a\left(\xi+\eta+\lambda \right)-a\left(\xi+\eta\right)-a\left(\xi+\lambda\right)+a\left(\xi\right)  \right|\lesssim |\eta||\lambda|\sup\limits_{|\xi'|\approx |\xi|} \left| a^{''}\left(\xi'\right) \right|.    
\end{align}
\begin{prop}\label{lem:multiplier estimate b_4}
	Let $a$ be a multiplier in $S^s_\epsilon$. Then for each dyadic $M_1\leq M_2\leq M_3$ there is an extension of $b_4$ from the diagonal set
	\begin{align*}
	\left\{\left(\xi_1,\xi_2,\xi_3,\xi_4 \right) \in P_4, \vert \xi_{1} \vert \approx M_1, \vert \xi_{2} \vert   \approx M_2, \vert \xi_{3} \vert,\vert \xi_4\vert \approx M_3 \right\}
	\end{align*} 
to the full dyadic set
\begin{align*}
\left\{   \vert \xi_{1} \vert \approx M_1, \vert \xi_{2} \vert   \approx M_2, \vert \xi_{3} \vert,\vert \xi_4\vert \approx M_3    \right\}
\end{align*} 
which satisfies the size and regularity conditions
	\begin{align}\label{eqn:multiplier estimate in energy estimate}
	\left\vert   \partial_{\xi_1}^{\beta_1}\partial_{\xi_2}^{\beta_2}\partial_{\xi_3}^{\beta_3} \partial_{\xi_4}^{\beta_4} b_4\left(\xi_1,\xi_2,\xi_3,\xi_4 \right)         \right\vert \lesssim    a\left(M_1\right)M_2^{-1}M_3^{-3} M_1^{-\beta_1}M_2^{-\beta_2}M_3^{-\beta_3-\beta_4}
	\end{align}
The implicit constants are independent of $M_1,M_2,M_3$.	
\end{prop} 
The proof of Proposition \ref{lem:multiplier estimate b_4} is analogous to the proof of the Proposition 5.2 in \cite{KT2007}. The only difference is that the stronger dispersion produces an extra smoothing effect as much as $M_3^ 2$. For reader's convenience, 
we present the proof in detail.  
\begin{proof}	
	Observe that on $P_4$ resonance function admits the following factorization
	\begin{align*}
\xi_1^4-\xi_2^4+\xi_3^4-\xi_4^4=\left(\xi_4-\xi_1\right)\left(\xi_4-\xi_3\right)\left(\xi_1^2+\xi_2^2+\xi_3^2+\xi_4^4+2\left(\xi_1+\xi_3\right)^2   \right)   
	\end{align*}
	along with all versions of it due to the symmetries of $P_4$. For the proof, see \cite{OT2016}. We consider several cases:
	
	(i) $M_1\ll M_2\leq M_3$. Then the extension of $b_4$ is defined using the formula
	\begin{align*}
	b_4\left(\xi_1,\xi_2,\xi_3,\xi_4\right)=\frac{a(\xi_1)-a(\xi_2)+a(\xi_3)-a(\xi_4)}{(\xi_4-\xi_1)(\xi_1-\xi_2) \left(\xi_1^2+\xi_2^2+\xi_3^2+\xi_4^4+2\left(\xi_1+\xi_3\right)^2   \right) }
	\end{align*}
	and its size and regularity properties are easily followed from $\vert \xi_4-\xi_1 \vert \approx M_3$ and $\vert \xi_1-\xi_2 \vert \approx M_2$.
	
	(ii) $M_1 \approx M_2 \ll M_3$. Then the extension of $b_4$ is defined by
	\begin{align*}
	b_4\left(\xi_1,\xi_2,\xi_3,\xi_4\right)=&\frac{a(\xi_1)-a(\xi_2)}{(\xi_1-\xi_2)(\xi_4-\xi_1) \left(\xi_1^2+\xi_2^2+\xi_3^2+\xi_4^4+2\left(\xi_1+\xi_3\right)^2   \right) }\\
	 -&\frac{a(\xi_3)-a(\xi_4)}{(\xi_3-\xi_4)(\xi_4-\xi_1) \left(\xi_1^2+\xi_2^2+\xi_3^2+\xi_4^4+2\left(\xi_1+\xi_3\right)^2   \right) }.
	\end{align*}
	Observe that $\vert \xi_4-\xi_1 \vert \approx M_3$ and the remaining quotients exhibits cancellation properties. More precisely, by using the mean value formula $\eqref{eqn: mean value theorem}$ , we have
	\begin{align*}
	\left\vert \frac{a(\xi_1)-a(\xi_2) }{\xi_1-\xi_2} \right\vert \lesssim   \frac{a\left(M_1\right)}{M_1} \quad \text{and} \quad \left\vert \frac{a(\xi_3)-a(\xi_4) }{\xi_3-\xi_4} \right\vert \lesssim   \frac{a\left(M_3\right)}{M_3}.
	\end{align*}
	
	(iii) $M_1\approx M_2 \approx M_3$. Then the extension of $b_4$ is defined by
	\begin{align*}
	b_4\left(\xi_1,\xi_2,\xi_3,\xi_4 \right)=- \frac{ a(\xi_4-(\xi_4-\xi_3)-(\xi_4-\xi_1)) -a(\xi_4-(\xi_4-\xi_3))-a(\xi_4-(\xi_4-\xi_1))+a(\xi_4)  }{\left(\xi_4-\xi_1\right)\left(\xi_4-\xi_3\right)\left(\xi_1^2+\xi_2^2+\xi_3^2+\xi_4^4+2\left(\xi_1+\xi_3\right)^2   \right)   }.
	\end{align*}
	In this case, the resonant interaction is the most serious. But we can also use the cancellation properties. More precisely, by using the double mean value theorem \eqref{eqn: double mean value theorem},  we have
	\begin{align*}
		\left\vert \frac{ a(\xi_4-(\xi_4-\xi_3)-(\xi_4-\xi_1)) -a(\xi_4-(\xi_4-\xi_3))-a(\xi_4-(\xi_4-\xi_1))+a(\xi_4)  }{\left(\xi_4-\xi_1\right)\left(\xi_4-\xi_3\right) } \right\vert\lesssim \frac{a(M_3)}{M_3^2}.
	\end{align*} 	
	
\end{proof}

The effect of $E_1$ to the modified energy is easily controlled by $E_0$.    
\begin{prop}\label{prop:E_1 energy} 
Let $a\in S_\epsilon^s$ with $  s+\epsilon <-\frac{1}{2}$. Then we have
\begin{align}\label{eqn:remaining in correction 2}
\left| E_1\left(u\right)  \right|\lesssim E_0\left(u\right)^2.
\end{align}		
\end{prop} 
\begin{proof}	
We may assume the functions $\widehat{u_j}$ are nonnegative. By using the Lemma \ref{lem:multiplier estimate b_4}, we have
\begin{align*}
\left|E_{1}\left(u\right)\right|=&\left|\int_{P_4} b_4\left(\xi_1,\dots,\xi_4\right) \widehat{u}\left(\xi_1\right)\overline{\widehat{u}}\left(\xi_2\right)\widehat{u}\left(\xi_3\right)\overline{\widehat{u}}\left(\xi_4\right)\right|\\
\lesssim & \sum\limits_{ \substack{ N_1,N_2,N_3,N_4: \\ \left\{N_1,N_2,N_3,N_4 \right\}=\left\{N_{\min},N_{\text{med}}, N_{\max}, N_{\max} \right\} }} \frac{a\left(N_{\min}\right)}{N_{\text{med}}N_{\max}^3}\Vert u_{N_1}u_{N_2}u_{N_3}u_{N_4}\Vert_{L_x^1}. 
\end{align*}
We may assume $N_1\leq N_2\leq N_3 \approx N_4$ by using the symmetry. Therefore, we focus on the summation:
\begin{align*}
&\sum\limits_{1\leq N_1\leq N_2 \leq N_3 \approx N_4} \frac{a\left(N_{1}\right)}{N_2N_3^3}\Vert u_{N_1}u_{N_2}u_{N_3}u_{N_4}\Vert_{L_x^1}\\
\lesssim& \sum\limits_{1\leq N_1\leq N_2 \leq N_3 \approx N_4} \frac{a\left(N_{1}\right)}{N_2N_3^3} \Vert u_{N_1}\Vert_{L_x^{\infty}} \Vert u_{N_2}\Vert_{L_x^{\infty}} \Vert u_{N_3}\Vert_{L_x^2} \Vert u_{N_4}\Vert_{L_x^2}.
\end{align*}
By using Bernstein's inequality, we have
\begin{align*}
&\sum\limits_{1\leq N_1\leq N_2 \leq N_3 \approx N_4} \frac{a\left(N_{1}\right)}{N_2N_3^3}N_1^{\frac{1}{2}}N_2^{\frac{1}{2}} \Vert u_{N_1}\Vert_{L_x^2} \Vert u_{N_2}\Vert_{L_x^2} \Vert u_{N_3}\Vert_{L_x^2} \Vert u_{N_4}\Vert_{L_x^2}\\
\lesssim&
E_0\left(u\right)^2 \sum\limits_{1\leq N_1\leq N_2 \leq N_3 \approx N_4} \frac{a\left(N_{1}\right)^\frac{1}{2}}{N_3^3}\frac{N_1^{\frac{1}{2}}N_2^{-\frac{1}{2}}}{a\left(N_2\right)^{\frac{1}{2}} a\left(N_3\right)  }.
\end{align*}
The remaining summation with respect to $N_1,N_2,N_3$ is easily handled. In fact, 
it is enough to assume $s>-\frac{7}{6}$.
\end{proof}

\begin{prop}\label{proposition: energy estimate 6 linear} Let $a \in S_\epsilon^s$ with $s+\epsilon <-\frac{1}{2}$ and $s>-\frac{3}{4}$.  Then we have
\begin{align*}
\left\vert \int\limits_0^1 R_6\left(u\right)\,dt \right\vert\lesssim \Vert u \Vert_{X^s}^6. 
\end{align*}
\end{prop} 
\begin{proof}
We consider a dyadic decomposition and represent the above integral in the frequency side as a dyadic sum of terms of the form
\begin{align*}
&\int_0^1\int\limits_{\substack{\xi_1-\xi_2+\xi_3=\xi \\ \xi_4-\xi_5+\xi_6=\xi \\ \vert \xi \vert \approx N} } b_4\left(\xi_1,\xi_2,\xi_3,\xi \right)\widehat{u_{N_1}}\left(\xi_1\right)\overline{\widehat{u_{N_2}}\left(\xi_2\right) }\widehat{u_{N_3}}\left(\xi_3\right)\varphi_N\left(\xi\right)\left( \overline{\widehat{u_{N_4}}\left(\xi_4\right) }      \widehat{u_{N_5}}\left(\xi_5\right) \overline{\widehat{u_{N_6}\ }}\left(\xi_6\right)     \right)  \,dt.
\end{align*}  	
Here $\varphi_N$ is the Fourier multiplier for $P_N$.
There are two cases to consider:

\textbf{Case 1:} $N\ll N_4,N_5,N_6$. Then for the frequency $N$ factor we take advantage of Lemma \ref{lem:improved trilinear}. We denote
\begin{align*}
\left\{ N,N_1,N_2,N_3 \right\}=\left\{M_1,M_2,M_3,M_3    \right\}, \quad M_1\leq M_2 \leq M_3,
\intertext{and}
f_{N}=\chi_I\sum\limits_{\substack{N_4,N_5,N_6: \\ N\ll N_4,N_5,N_6 }} P_{N}\left( \overline{u_{N_4}}      u_{N_5} \overline{u_{N_6}\ }     \right), \quad \vert I \vert=N^{4s+2}.
\end{align*}
Since $b_4$ is smooth in each variable on the corresponding dyadic scale, we can expand it into a rapidly convergent Fourier series. This allows us to separate variables and reduce the problem to the case when $b_4$ is of product type 
\begin{align*} 
b_4\left(\xi_1,\xi_2,\xi_3,\xi\right)=\frac{a\left(M_1\right)}{M_2 M_3^3} \chi^1\left(\xi_1\right)\chi^2\left(\xi_2\right)\chi^3\left(\xi_3\right)\chi^0\left(\xi\right), 
\end{align*}
where $\chi^j$'s are unit size bump functions which are smooth on the respective dyadic scales . Since the symbol $\chi^j\left(D\right)$ are bounded in $U_{S}^2$ space, we can discard $\chi^1,\chi^2$ and $\chi^3$ and incorporate $\chi^0$ into $P_{N}$.
In the following, we drop the complex conjugate sign.
Therefore, we have reduced the problem to the case 
\begin{align*}
&\sum\limits_{\substack{ \left\{N,N_1,N_2,N_3\right\}\\=\left\{M_1,M_2,M_3,M_3 \right\} } } \frac{a\left(M_1\right)}{M_2M_3^{3}}    \int_0^1\int\limits_\mathbb{R} u_{N_1}u_{N_2}u_{N_3}  \sum\limits_{\substack{N_4,N_5,N_6: \\ N\ll N_4,N_5,N_6 }} P_{N}\left( u_{N_4}      u_{N_5} u_{N_6}      \right)  \,dx \,dt\\
=&\sum\limits_{\substack{ \left\{N,N_1,N_2,N_3\right\}\\=\left\{M_1,M_2,M_3,M_3 \right\} } } \frac{a\left(M_1\right)}{M_2M_3^{3}}  \sum\limits_{\substack{I\subset [0,1] \\ \vert I \vert=N^{4s+2} }}  \int\limits_I \int\limits_\mathbb{R} u_{N_1}u_{N_2}u_{N_3}  f_N  \,dx \,dt\\
=&\sum\limits_{\substack{ \left\{N,N_1,N_2,N_3\right\}\\=\left\{M_1,M_2,M_3,M_3 \right\} } } \Lambda_1.
\end{align*}	
We decompose $f_N$ into low modulation output and high modulation output
\begin{align*}
Q_{\lesssim M_3^4} \sum\limits_{N\ll N_4,N_5,N_6} \chi_IP_N\left(u_{N_4}u_{N_5}u_{N_6}\right)+Q_{\gg M_3^4 } \sum\limits_{N\ll N_4,N_5,N_6}\chi_I P_N\left(u_{N_4}u_{N_5}u_{N_6}\right)
\end{align*} 		 
 
\textbf{Subcase 1.a} $N=M_3$. First, we consider the low modulation output.
For the $L_t^1L_x^2$ term in $f_N$, we estimate $u_{M_1},u_{M_2}$ in $L_{t,x}^{\infty}$ and $u_{M_3}$ in $L_t^\infty L_x^2$ by using the Bernstein's inequality and Lemma \ref{lem:improved trilinear}:
\begin{align*}
\vert \Lambda_1 \vert \lesssim&  \frac{a\left(M_1\right)}{M_2M_3^{3}}  M_3^{-4s-2} \sup\limits_{\substack{I\subset [0,1]\\ \vert I \vert=M_3^{4s+2}}} \left\vert \int\limits_I \int\limits_\mathbb{R} u_{M_1}u_{M_2}u_{M_3} Q_{\lesssim M_3^4}f_N \,dx\,dt \right\vert\\
\lesssim&\frac{a\left(M_1\right)}{M_2M_3^{3}}  M_3^{-4s-2} \sup\limits_{\substack{I\subset [0,1] \\ \vert I \vert=M_3^{4s+2} }} \Vert \chi_I u_{M_1} \Vert_{L_{t,x}^\infty}\Vert \chi_I u_{M_2} \Vert_{L_{t,x}^{\infty}}\Vert \chi_I u_{M_3}\Vert_{L_t^\infty L_x^2}\Vert Q_{\lesssim M_3^4 }f_N \Vert_{L_t^1L_x^2}\\
\lesssim& \frac{a\left(M_1\right)}{M_2M_3^{3}} M_3^{-4s-2}\left(M_1^{\frac{1}{2}} M_2^{\frac{1}{2}} \right)\left(M_3^{-s}M_2^{-s}M_1^{-s} \right)M_3^{-3-3s} \Vert u_{M_1} \Vert_{X^s}  \Vert u_{M_2} \Vert_{X^s}  \Vert u_{M_3} \Vert_{X^s} \Vert u \Vert_{X^s}^3\\
\lesssim&  M_3^{-8s-8} a\left(M_1\right) M_1^{-s+\frac{1}{2}}M_2^{-\frac{1}{2}-s}\Vert u_{M_1} \Vert_{X^s}  \Vert u_{M_2} \Vert_{X^s}  \Vert u_{M_3} \Vert_{X^s} \Vert u \Vert_{X^s}^3.
\end{align*}
Therefore, we consider the summation
\begin{align*} 
 \sum\limits_{M_1\leq M_2 \leq M_3 }    M_3^{-8s-8}a\left(M_1\right)M_1^{-s+\frac{1}{2}}M_2^{-\frac{1}{2}-s}\Vert u_{M_1} \Vert_{X^s}  \Vert u_{M_2} \Vert_{X^s}  \Vert u_{M_3} \Vert_{X^s} \Vert u \Vert_{X^s}^3. 
\end{align*} 
This summation can be dealt with $s\geq-\frac{17}{18}$. 

For the $L_x^{\frac{4}{3} }L_t^1$ term in $f_N$, we estimate $u_{M_1},u_{M_2}$ in $L^\infty$ and $u_{M_3}$ in $L_x^4L_t^{\infty}$. Then by using Bernstein's inequality, maximal function estimate and Lemma \ref{lem:improved trilinear}, we obtain
\begin{align*}
\vert \Lambda_1 \vert \lesssim&  \frac{a\left(M_1\right)}{M_2M_3^{3}}  M_3^{-4s-2} \sup\limits_{\substack{I\subset [0,1]\\ \vert I \vert=M_3^{4s+2}}} \left\vert \int\limits_I \int\limits_\mathbb{R} u_{M_1}u_{M_2}u_{M_3} Q_{\lesssim M_3^4}f_N \,dx\,dt \right\vert\\
\lesssim& \frac{a\left(M_1\right)}{M_2M_3^{3}}  M_3^{-4s-2} \sup\limits_{\substack{I\subset [0,1] \\ \vert I \vert=M_3^{4s+2} }} \Vert \chi_I u_{M_1} \Vert_{L_{t,x}^\infty}\Vert \chi_I u_{M_2} \Vert_{L_{t,x}^{\infty}}\Vert \chi_I u_{M_3}\Vert_{L_x^4 L_t^\infty}\Vert Q_{\lesssim M_3^4 }f_N \Vert_{L_x^{\frac{4}{3}}L_t^1}\\
\lesssim & \frac{a\left(M_1\right)}{M_2M_3^{3}} M_3^{-4s-2}\left(M_1^{\frac{1}{2}} M_2^{\frac{1}{2}} M_3^{\frac{1}{4}} \right)\left(M_3^{-s}M_2^{-s}M_1^{-s} \right)M_3^{-\frac{13}{4}-3s} \Vert u_{M_1} \Vert_{X^s}  \Vert u_{M_2} \Vert_{X^s}  \Vert u_{M_3} \Vert_{X^s} \Vert u \Vert_{X^s}^3,
\end{align*}
which gives the same result as in the previous case. Therefore, the summation with respect to $M_1,M_2,M_3$ can be dealt with $s\geq-\frac{17}{18}$. 
 
For the high modulation part of $f_N$ at modulation $\sigma \gg M_3^4 $, we observe that at least one of three factors $\chi_I u_{M_1}, \chi_Iu_{M_2}, \chi_I u_{M_3}$ must have modulation at least $\sigma \gg M_3^4 $. We may assume $Q_{\sigma }\left(\chi_I u_{M_1}\right)$. This is the worst case. We bound $Q_{\sigma}\chi_Iu_{M_1}$ in $L^2$ and the other two $u_{M_2},u_{M_3}$ in $L^\infty$. Observe that
\begin{align*}  
\vert \Lambda_1 \vert \lesssim&  \frac{a\left(M_1\right)}{M_2M_3^{3}}  M_3^{-4s-2} \sup\limits_{\substack{I\subset [0,1]\\ \vert I \vert=M_3^{4s+2}}} \left\vert \sum\limits_{\sigma \gg M_3^4} \int\limits_\mathbb{R} \int\limits_\mathbb{R}  \chi_Iu_{M_1}\chi_Iu_{M_2} \chi_Iu_{M_3} Q_{\sigma
	 }f_N \,dx\,dt \right\vert\\
 \lesssim &  \frac{a\left(M_1\right)}{M_2M_3^{3}}  M_3^{-4s-2} \left(M_2^{\frac{1}{2}} M_3^{\frac{1}{2}}\right)\left(M_1^{-s}M_2^{-s}M_3^{-s} \right)M_3^{-3-3s}\Vert u_{M_1} \Vert_{X^s} \Vert u_{M_2} \Vert_{X^s} \Vert u_{M_3} \Vert_{X^s}          \Vert u \Vert_{X^s}^3\\
 \lesssim & a\left(M_1\right)M_1^{-s}M_2^{-\frac{1}{2}-s}M_3^{-8s-\frac{15}{2}} \Vert u_{M_1} \Vert_{X^s} \Vert u_{M_2} \Vert_{X^s} \Vert u_{M_3} \Vert_{X^s}          \Vert u \Vert_{X^s}^3.
\end{align*} 
The summation with respect to $M_1,M_2,M_3$ is handled if $s\geq-\frac{8}{9}$.

\textbf{Subcase 1.b} $N=M_2\ll M_3$. First, we consider the low modulation output.
 For the $L_t^1L_x^2$ term in $ Q_{\lesssim M_3^4 }f_N$, we estimate $u_{M_1}u_{M_3}u_{M_3}$ in $L_{t,x}^2$. By considering the interval summation loss and using the Bernstein inequality, Lemma \ref{lem:improved trilinear}, we have
 \begin{align*}
 \vert \Lambda_1 \vert \lesssim&  \frac{a\left(M_1\right)}{M_2M_3^{3}}  M_2^{-4s-2} \sup\limits_{\substack{I\subset [0,1]\\ \vert I \vert=M_2^{4s+2}}} \left\vert \int\limits_I \int\limits_\mathbb{R} u_{M_1}u_{M_3}u_{M_3} Q_{\lesssim M_3^4}f_N \,dx\,dt \right\vert\\
 \lesssim&\frac{a\left(M_1\right)}{M_2M_3^{3}}  M_2^{-4s-2} \sup\limits_{\substack{I\subset [0,1] \\ \vert I \vert=M_2^{4s+2} }} \Vert \chi_Iu_{M_1}u_{M_3}u_{M_3} \Vert_{L_t^2 L_x^2\left(I\times \mathbb{R}\right)} \Vert Q_{\lesssim M_3^4 }f_N \Vert_{L_t^2L_x^2}\\
 \lesssim&  \frac{a\left(M_1\right)}{M_2M_3^{3}}  M_2^{-4s-2}  \sup\limits_{\substack{I\subset [0,1] \\ \vert I \vert=M_2^{4s+2} }} \left( \sum\limits_{\substack{J\subset I\\ \vert J \vert =M_3^{4s+2}}}\Vert \chi_J u_{M_1}u_{M_3}u_{M_3} \Vert^2_{L_t^2 L_x^2\left(J\times \mathbb{R}\right)} \right)^{\frac{1}{2}} M_3^2 \Vert Q_{\lesssim M_3^4 }f_N \Vert_{L_t^1L_x^2}\\
 \lesssim& \frac{a\left(M_1\right)}{M_2M_3^{3}}  M_2^{-4s-2} \left( M_2^{4s+2}M_3^{-4s-2} \right)^{\frac{1}{2}}  \sup\limits_{\substack{J\subset [0,1] \\ \vert J \vert=M_3^{4s+2} }} \Vert \chi_J u_{M_1} u_{M_3} \Vert_{L_{t,x}^2}\Vert \chi_J u_{M_3} \Vert_{L_{t,x}^\infty} M_3^2 M_2^{-3-3s} \Vert u\Vert_{X^s}^3.
 \end{align*}
 Hence, by using the bilinear estimates and Bernstein's inequality, we have
 \begin{align*}
\vert \Lambda_1 \vert  \lesssim&\frac{a\left(M_1\right)}{M_2M_3^{3}} \left(M_2^{-2s-1}M_3^{-2s-1} \right) \left(M_1^{-s}M_3^{-s}M_3^{-s} \right)M_3^{-\frac{3}{2}} M_3^{\frac{1}{2}}M_3^2 M_2^{-3-3s} \Vert u_{M_3} \Vert_{X^s} \Vert u_{M_3} \Vert_{X^s} \Vert u \Vert_{X^s}^4 \\
 \lesssim & a\left(M_1\right)M_1^{-s} M_2^{-5s-5}M_3^{-4s-3} \Vert u_{M_3} \Vert_{X^s} \Vert u_{M_3} \Vert_{X^s} \Vert u \Vert_{X^s}^4.
 \end{align*}
 Therefore, the summation with respect to $M_1,M_2,M_3$ is handled if $s\geq - \frac{3}{4}$.

 Next, we consider the $L_x^{\frac{4}{3} }L_t^1$ term in $f_N$. By considering the interval summation loss and using the Bernstein's inequality, Lemma \ref{lem:improved trilinear}, we have
 \begin{align*}
 \vert \Lambda_1 \vert \lesssim&  \frac{a\left(M_1\right)}{M_2M_3^{3}}  M_2^{-4s-2} \sup\limits_{\substack{I\subset [0,1]\\ \vert I \vert=M_2^{4s+2}}} \left\vert \int\limits_I \int\limits_\mathbb{R} u_{M_1}u_{M_3}u_{M_3} Q_{\lesssim M_3^4}f_N \,dx\,dt \right\vert\\
 \lesssim&\frac{a\left(M_1\right)}{M_2M_3^{3}}  M_2^{-4s-2} \sup\limits_{\substack{I\subset [0,1] \\ \vert I \vert=M_2^{4s+2} }} \Vert \chi_Iu_{M_1}u_{M_3}u_{M_3} \Vert_{L_t^2 L_x^2\left(I\times \mathbb{R}\right)} \Vert Q_{\lesssim M_3^4 }f_N \Vert_{L_t^2L_x^2}\\
 \lesssim&  \frac{a\left(M_1\right)}{M_2M_3^{3}}  M_2^{-4s-2}  \left( M_2^{4s+2}M_3^{-4s-2} \right)^{\frac{1}{2}}&\\
 \times&  \sup\limits_{\substack{J\subset [0,1] \\ \vert J \vert=M_3^{4s+2} }} \Vert \chi_J u_{M_1} u_{M_3} \Vert_{L_{t,x}^2}\Vert \chi_J u_{M_3} \Vert_{L_{t,x}^\infty} \left(M_3^4\right)^\frac{1}{2}N^\frac{1}{4} \Vert Q_{\lesssim M_3^4 }f_N \Vert_{L_x^\frac{4}{3}L_t^1}\\
 \lesssim& \frac{a\left(M_1\right)}{M_2M_3^{3}}  M_2^{-4s-2} \left( M_2^{4s+2}M_3^{-4s-2} \right)^{\frac{1}{2}}&\\
   &\times \sup\limits_{\substack{J\subset [0,1] \\ \vert J \vert=M_3^{4s+2} }} \Vert \chi_J u_{M_1} u_{M_3} \Vert_{L_{t,x}^2}\Vert \chi_J u_{M_3} \Vert_{L_{t,x}^\infty} M_3^2 M_2^{\frac{1}{4}} M_2^{-\frac{13}{4}-3s} \Vert u\Vert_{X^s}^3.
 \end{align*}
  Hence, by using the bilinear estimates and Bernstein inequality, we have
 \begin{align*}
 \vert \Lambda_1 \vert \lesssim  a\left(M_1\right)M_1^{-s} M_2^{-5s-5}M_3^{-4s-3} \Vert u_{M_3} \Vert_{X^s} \Vert u_{M_3} \Vert_{X^s} \Vert u \Vert_{X^s}^4.
 \end{align*}
 This summation with respect to $M_1,M_2,M_3$ is also handled if $s\geq - \frac{3}{4}$.

 For the high modulation part of $f_N$ at modulation $\sigma \gg M_3^4 $, we observe that one of three factors $\chi_I u_{M_1}, \chi_I u_{M_3}, \chi_Iu_{M_3}$ must have modulation at least $\sigma $. We may assume $Q_{\sigma}\left(\chi_Iu_{M_1}\right)$. This is the worst case. We bound $Q_{\sigma}\left( \chi_Iu_{M_1}\right)$ in $L^2$ and the other two $\chi_Iu_{M_2},\chi_Iu_{M_3}$ in $L^\infty$. Observe that
 \begin{align*}
 \vert \Lambda_1 \vert \lesssim&  \frac{a\left(M_1\right)}{M_2M_3^{3}}  M_2^{-4s-2} \sup\limits_{\substack{I\subset [0,1]\\ \vert I \vert=M_2^{4s+2}}} \left\vert  \sum\limits_{\sigma \gg M_3^4}\int\limits_\mathbb{R} \int\limits_\mathbb{R} \chi_Iu_{M_1}\chi_Iu_{M_3}\chi_Iu_{M_3} Q_{\sigma }f_N \,dx\,dt \right\vert\\
 \lesssim&\frac{a\left(M_1\right)}{M_2M_3^{3}}  M_2^{-4s-2} \sup\limits_{\substack{I\subset [0,1] \\ \vert I \vert=M_2^{4s+2} }} \sum\limits_{\sigma \gg M_3^4}\Vert Q_{\sigma}\left(\chi_I u_{M_1}\right) \chi_Iu_{M_3}\chi_Iu_{M_3} \Vert_{L_t^2 L_x^2\left(I\times \mathbb{R}\right)} \Vert Q_{\sigma }f_N \Vert_{L_t^2L_x^2}\\
 \lesssim&  \frac{a\left(M_1\right)}{M_2M_3^{3}}  M_2^{-4s-2}  \sup\limits_{\substack{I\subset [0,1] \\ \vert I \vert=M_2^{4s+2} }}\sum\limits_{\sigma \gg M_3^4} \left( \sum\limits_{\substack{J\subset I\\ \vert J \vert =M_3^{4s+2}}}\Vert Q_\sigma \left(\chi_I u_{M_1}\right) u_{M_3}u_{M_3} \Vert^2_{L_t^2 L_x^2\left(J\times \mathbb{R}\right)} \right)^{\frac{1}{2}} \Vert Q_{\sigma  }f_N \Vert_{L_t^2L_x^2}\\
 \lesssim& \frac{a\left(M_1\right)}{M_2M_3^{3}} \left(M_2^{-2s-1}M_3^{-2s-1} \right)& \\
 &\times \sup\limits_{\substack{I\subset [0,1]\\ \vert I \vert=M_2^{4s+2
 }}} \sum\limits_{\sigma \gg M_3^4}\sup\limits_{\substack{J\subset I \\ \vert J \vert=M_3^{4s+2} }} \Vert Q_{\sigma}\left( \chi_I u_{M_1}\right) \Vert_{L_{t,x}^2} \Vert \chi_J u_{M_3} \Vert_{L_{t,x}^\infty} \Vert \chi_J u_{M_3} \Vert_{L_{t,x}^\infty}  \Vert Q_{\sigma  }f_N \Vert_{L_t^2L_x^2}.
 \end{align*} 
 Hence, by using Bernstein inequality, we have  
\begin{align*}
\vert \Lambda_1 \vert  \lesssim&\frac{a\left(M_1\right)}{M_2M_3^{3}} \left(M_2^{-2s-1}M_3^{-2s-1} \right) \left(M_1^{-s}M_3^{-s}M_3^{-s} \right) M_3 M_2^{-3-3s} \Vert u_{M_3} \Vert_{X^s} \Vert u_{M_3} \Vert_{X^s} \Vert u \Vert_{X^s}^4 \\
\lesssim & a\left(M_1\right)M_1^{-s} M_2^{-5s-5}M_3^{-4s-3} \Vert u_{M_3} \Vert_{X^s} \Vert u_{M_3} \Vert_{X^s} \Vert u \Vert_{X^s}^4.
\end{align*}
Hence, the summation with respect to $M_1,M_2,M_3$ is also handled if $s\geq - \frac{3}{4}$.

\textbf{Subcase 1.c} $N=M_1\ll M_2$. First, we consider the low modulation output.
For the $L_t^1L_x^2$ term in $f_N$, we estimate $u_{M_2}u_{M_3}u_{M_3}$ in $L_{t,x}^2$. By considering the interval summation loss with square summability and using Lemma \ref{lem:improved trilinear}, we have
\begin{align*}
\vert \Lambda_1 \vert \lesssim&  \frac{a\left(M_1\right)}{M_2M_3^{3}}  M_1^{-4s-2} \sup\limits_{\substack{I\subset [0,1]\\ \vert I \vert=M_1^{4s+2}}} \left\vert \int\limits_I \int\limits_\mathbb{R} u_{M_2}u_{M_3}u_{M_3} Q_{\lesssim M_3^4}f_N \,dx\,dt \right\vert\\
\lesssim&\frac{a\left(M_1\right)}{M_2M_3^{3}}  M_1^{-4s-2} \sup\limits_{\substack{I\subset [0,1] \\ \vert I \vert=M_1^{4s+2} }} \Vert \chi_Iu_{M_2}u_{M_3}u_{M_3} \Vert_{L_t^2 L_x^2\left(I\times \mathbb{R}\right)} \Vert Q_{\lesssim M_3^4 }f_N \Vert_{L_t^2L_x^2}\\
\lesssim&  \frac{a\left(M_1\right)}{M_2M_3^{3}}  M_1^{-2s-1}M_3^{-2s-1}   \sup\limits_{\substack{J\subset [0,1] \\ \vert J \vert=M_3^{4s+2} }} \Vert \chi_J u_{M_2} u_{M_3} \Vert_{L_{t,x}^2}\Vert \chi_J u_{M_3} \Vert_{L_{t,x}^\infty}
 M_3^2 \Vert Q_{\lesssim M_3^4 }f_N \Vert_{L_t^1L_x^2}.
\end{align*}
Observe that even if $M_2 \approx M_3$ as two of the $M_3$ sized frequencies should be $M_3$ separatd in order for output frequency to be localized at $N$. Therefore, by using the bilinear estimates and Lemma \ref{lem:improved trilinear}, we have
\begin{align*}
\vert \Lambda_1 \vert  \lesssim&\frac{a\left(M_1\right)}{M_2M_3^{3}} \left(M_1^{-2s-1}M_3^{-2s-1} \right) \left(M_2^{-s}M_3^{-s}M_3^{-s} \right)M_3^{-\frac{3}{2}} M_3^{\frac{1}{2}} M_3^2 M_1^{-3-3s} \Vert u_{M_3} \Vert_{X^s} \Vert u_{M_3} \Vert_{X^s} \Vert u \Vert_{X^s}^4 \\
\lesssim & a\left(M_1\right)M_1^{-5s-4} M_2^{-1-s}M_3^{-4s-3} \Vert u_{M_3} \Vert_{X^s} \Vert u_{M_3} \Vert_{X^s} \Vert u \Vert_{X^s}^4.
\end{align*}
Therefore, the summation with respect to $M_1,M_2,M_3$ is handled if $s \geq -\frac{3}{4}$.

For the $L_x^{\frac{4}{3} }L_t^1$ term in $f_N$, we can proceed as in Subcase 1.b. In this case, the summation with respect to $M_1,M_2,M_3$ is handled if $s \geq -\frac{3}{4}$.

For the high modulation part of $f_N$ at modulation $\sigma \gg M_3^4 $, we observe that one of three factors $u_{M_2}, u_{M_3},u_{M_3}$ must have modulation at least $\sigma $. We may assume $Q_{\sigma}u_{M_2}$. This is the worst case. We bound $Q_{\sigma}u_{M_2}$ in $L^2$ and the other two $u_{M_3},u_{M_3}$ in $L^\infty$. By proceeding as in Subcase 1.b, we obtain 
\begin{align*} 
\vert \Lambda_1 \vert  
\lesssim &\; a\left(M_1\right) M_1^{-5s-4}M_2^{-1-s}M_3^{-4s-3}  \Vert u_{M_3} \Vert_{X^s} \Vert u_{M_3} \Vert_{X^s} \Vert u \Vert_{X^s}^4.
\end{align*} 
Hence, the summation with respect to $M_1,M_2,M_3$ is also handled if $s\geq - \frac{3}{4}$.
 
\textbf{Case 2.} $N\gtrsim \min\left\{ N_4,N_5,N_6 \right\}$. Without loss of generality we may assume 
\begin{align*}
N_1\leq N_2\leq N_3, \quad\quad N_4\leq N_5\leq N_6.
\end{align*}
Then we must have
\begin{align*} 
N_4 \lesssim N \lesssim N_3 
\end{align*}
We denote
\begin{align*}
\left\{ N,N_1,N_2,N_3 \right\}=\left\{M_1,M_2,M_3,M_3    \right\}, \quad M_1\leq M_2 \leq M_3.
\end{align*} 
We may expand the Fourier multiplier $\varphi_N$ for $P_N$ into a Fourier integral. For a Schwartz function $\rho_N$, we have
\begin{align}\label{eqn:Fourier multiplier decomposition}
\varphi_N(\xi)=\int \rho_N(y)e^{i\xi y}\,dy=\int \rho_N(y) e^{i\xi_4y}e^{-i\xi_5y}e^{i\xi_6y}\,dy ,\quad \xi=\xi_4-\xi_5+\xi_6. 
\end{align}
Here we can separate the exponential into three factors since in the domain of integration we have $\xi=\xi_4-\xi_5+\xi_6.$ The complex exponentials are bounded symbols and thus bounded on $U_S^2$. Therefore it can be harmlessly absorbed into $u_{N_4},u_{N_5},u_{N_6}$. Moreover we have $\Vert \rho_N \Vert_{L^1}\lesssim 1$ uniformly in $N$. Plugging in the expression $\eqref{eqn:multiplier estimate in energy estimate}$ and absorbing the factors 
originating from $\eqref{eqn:Fourier multiplier decomposition}$ into the $\widehat{u_i}$, we are left with estimating 
\begin{align*}
&\sum\limits_{\substack{ N,N_1,N_2,N_3,N_4,N_5,N_6 : \\ \left\{N,N_1,N_2,N_3\right\} =\left\{M_1,M_2,M_3,M_3  \right\} }} \frac{a\left(M_1\right)}{M_2M_3^{3}}    \int_0^1\int\limits_\mathbb{R} u_{N_1}u_{N_2}u_{N_3} u_{N_4}      u_{N_5} u_{N_6}       \,dx \,dt\\
=&\sum\limits_{\substack{ N,N_1,N_2,N_3,N_4,N_5,N_6 : \\ \left\{N,N_1,N_2,N_3\right\} =\left\{M_1,M_2,M_3,M_3  \right\} }} \Lambda_2. 
\end{align*}	 

\textbf{Subcase 2.a} $N=M_3$. In this case we have $N_4\lesssim N=M_3,  N=M_3\lesssim N_6$.

\textbf{Subcase 2.a.i} $N_6\approx N_5\gg M_3$. We use the bilinear estimates for the products $u_{M_3}u_{N_5}$ and $u_{N_4}u_{N_6}$ and the $L^\infty$ bound for $u_{M_1},u_{M_2}$. Then by considering the interval summation loss $N_6^{-4s-2}$, we have
\begin{align*}
\vert \Lambda_2 \vert \lesssim& \frac{a\left(M_1\right)}{M_2M_3^{3}} N_6^{-4s-2}\left(M_1^{-s}M_2^{-s}M_3^{-s} N_4^{-s}N_6^{-2s} \right)\left(M_1^{\frac{1}{2}} M_2^{\frac{1}{2}}\right)\left(N_6^{-\frac{3}{2}} N_6^{-\frac{3}{2}} \right)\prod\limits_{j=1}^3\Vert u_{M_j} \Vert_{X^s} \Vert u_{N_4} \Vert_{X^s} \Vert u_{N_6} \Vert_{X^s}^2\\
\lesssim& a\left(M_1\right)M_1^{\frac{1}{2}-s}M_2^{-\frac{1}{2}-s}M_3^{-s-3}N_4^{-s}N_6^{-6s-5}\prod\limits_{j=1}^3\Vert u_{M_j} \Vert_{X^s} \Vert u_{N_4} \Vert_{X^s}  \Vert u_{N_6} \Vert_{X^s}^2.
\end{align*}
Hence the summation with respect to $N,N_1,\dots,N_6$ is handled if $s\geq - \frac{5}{6}$.

\textbf{Subcase 2.a.ii} $N_5\leq N_6 \approx M_3$ and $M_2\ll M_3$. Then we use the bilinear estimates for $u_{M_2}u_{M_3}$ and $L^\infty$ bound for $u_{M_1}$ and the $L^6$ Strichartz estimate for the $u_{N_4},u_{N_5},u_{N_6}$. By considering the interval summation loss, we have
\begin{align*}
\vert \Lambda_2 \vert \lesssim& \frac{a\left(M_1\right)}{M_2M_3^{3}} M_3^{-4s-2}\left(M_1^{-s}M_2^{-s}M_3^{-s} N_4^{-s}N_5^{-s}M_3^{-s} \right)   M_3^{-\frac{3}{2}}M_1^{\frac{1}{2}}  \left(N_4^{-\frac{1}{3}} N_5^{-\frac{1}{3}} M_3^{-\frac{1}{3}} \right)\\
&\quad\quad\quad\quad \times                        
\prod\limits_{j=1}^3\Vert u_{M_j} \Vert_{X^s} \Vert u_{N_4} \Vert_{X^s} \Vert u_{N_5} \Vert_{X^s} \Vert u_{M_3} \Vert_{X^s}\\
\lesssim & a\left(M_1\right)M_1^{\frac{1}{2}-s}M_2^{-1-s}M_3^{-6s-\frac{41}{6}}N_4^{-s-\frac{1}{3}}N_5^{-s-\frac{1}{3}} \Vert u_{M_3} \Vert_{X^s} \Vert u_{M_3} \Vert_{X^s} \Vert u\Vert_{X^s}^4.
\end{align*} 
Hence the summation with respect to $N,N_1,\dots,N_6$ is handled if $s\geq - \frac{15}{16}$.

\textbf{Subcase 2.a.iii} $N_5\leq N_6 \approx M_3$ and $M_2\approx M_3$. In this case, we use the $L^6$ Strichartz estimate for all the factors. Then we have
\begin{align*}
\vert \Lambda_2 \vert \lesssim& \frac{a\left(M_1\right)}{M_2M_3^{3}} M_3^{-4s-2}\left(M_1^{-s}M_3^{-2s} N_4^{-s}N_5^{-s}M_3^{-s} \right)     \left(M_1^{-\frac{1}{3}} M_3^{-\frac{2}{3}} N_4^{-\frac{1}{3}}N_5^{-\frac{1}{3}}M_3^{-\frac{1}{3}} \right)  \prod\limits_{j=1}^6\Vert u_{M_j} \Vert_{X^s}\\
\lesssim & a\left(M_1\right) M_1^{-s-\frac{1}{3}}N_4^{-s-\frac{1}{3}}N_5^{-s-\frac{1}{3}}M_3^{-7s-7} \Vert u_{M_3} \Vert_{X^s} \Vert u_{M_3} \Vert_{X^s} \Vert u\Vert_{X^s}^4.
\end{align*}
Hence the summation with respect to $N,N_1,\dots,N_6$ is handled if $s\geq - \frac{23}{27}$. 

\textbf{Subcase 2.b} $N=M_2\ll M_3$. In this case we have $N_4\lesssim N=M_2$.   
 
\textbf{Subcase 2.b.i} $N_6\approx N_5\gg M_3$. In this case, we use the bilinear estimates for the products $u_{M_3}u_{N_5}$ and $u_{N_4}u_{N_6}$ and the $L^\infty$ bound for $u_{M_1},u_{M_3}$. Then we have
\begin{align*}
 \vert \Lambda_2 \vert \lesssim& \frac{a\left(M_1\right)}{M_2M_3^{3}} N_6^{-4s-2}\left(M_1^{-s}M_3^{-2s} N_4^{-s}N_6^{-2s} \right)\left(M_1^{\frac{1}{2}} M_3^{\frac{1}{2}}\right)\left(N_6^{-\frac{3}{2}} N_6^{-\frac{3}{2}} \right) \Vert u \Vert_{X^s}^4 \Vert u_{N_6} \Vert_{X^s}^2\\
 \lesssim& a\left(M_1\right) M_1^{-s+\frac{1}{2}}M_2^{-1}M_3^{-2s-\frac{5}{2}}N_4^{-s}N_6^{-6s-5} \Vert u \Vert_{X^s}^4 \Vert u_{N_6} \Vert_{X^s}^2.
 \end{align*} 
 
 Hence the summation with respect to $N,N_1,\dots,N_6$ is handled if $s\geq - \frac{5}{6}$.
 
 \textbf{Subcase 2.b.ii} $N_5\leq N_6 \approx M_3$. Then we use the bilinear estimates for $u_{M_1}u_{M_3}$ and $L^\infty$ bound for $u_{M_3}$ and the $L^6$ Strichartz estimate for the $u_{N_4},u_{N_5},u_{N_6}$. Then we have
 \begin{align*}
 \vert \Lambda_2 \vert \lesssim& \frac{a\left(M_1\right)}{M_2M_3^{3}} M_3^{-4s-2}\left(M_1^{-s}M_3^{-2s} N_4^{-s}N_5^{-s}M_3^{-s} \right)   M_3^{-\frac{3}{2}}M_3^{\frac{1}{2}}  \left(N_4^{-\frac{1}{3}} N_5^{-\frac{1}{3}} M_3^{-\frac{1}{3}} \right)                    
 \Vert u \Vert_{X^s}^4 \Vert u_{N_6} \Vert_{X^s}^2\\
 \lesssim & a\left(M_1\right)M_1^{-s}M_2^{-1}M_3^{-7s-\frac{19}{3}}N_4^{-s-\frac{1}{3}}N_5^{-s-\frac{1}{3}} \Vert u_{M_3} \Vert_{X^s} \Vert u_{M_3} \Vert_{X^s} \Vert u\Vert_{X^s}^4.
 \end{align*} 
 Hence the summation with respect to $N,N_1,\dots,N_6$ is handled if $s\geq - \frac{5}{6}$.
 
\textbf{Subcase 2.b.iii} $N_5\leq N_6 \ll M_3$.  In this case, we use the bilinear estimates for the products $u_{M_3}u_{N_5}$ and $u_{M_3}u_{N_6}$ and the $L^\infty$ bound for $u_{M_1},u_{N_4}$. Then we have 
\begin{align*}
 \vert \Lambda_2 \vert \lesssim& \frac{a\left(M_1\right)}{M_2M_3^{3}} M_3^{-4s-2}\left(M_1^{-s}M_3^{-2s} N_4^{-s}N_5^{-s}N_6^{-s} \right)     \left(M_1^{\frac{1}{2}} N_4^{\frac{1}{2}}\right)\left(M_3^{-\frac{3}{2}} M_3^{-\frac{3}{2}} \right) \Vert u \Vert_{X^s}^4 \Vert u_{M_3} \Vert_{X^s}^2 \\
 \lesssim & a\left(M_1\right)M_1^{-s+\frac{1}{2}}M_2^{-1}N_4^{-s+\frac{1}{2}}N_5^{-s}N_6^{-s}M_3^{-6s-8} \Vert u_{M_3} \Vert_{X^s} \Vert u_{M_3} \Vert_{X^s} \Vert u\Vert_{X^s}^4.
 \end{align*}
 Hence the summation with respect to $N,N_1,\dots,N_6$ is handled if $s\geq - \frac{17}{18}$.

\textbf{Subcase 2.c} $N=M_1\ll M_2$.    
In this case we have $N_4\lesssim N=M_1$.

\textbf{Subcase 2.c.i} $N_6\approx N_5\gg M_3$. We use the bilinear estimates for the products $u_{M_3}u_{N_5}$ and $u_{M_3}u_{N_6}$ and the $L^\infty$ bound for $u_{M_2},u_{N_4}$. Then we have
\begin{align*}
\vert \Lambda_2 \vert \lesssim& \frac{a\left(M_1\right)}{M_2M_3^{3}} N_6^{-4s-2}\left(M_2^{-s}M_3^{-2s} N_4^{-s}N_6^{-2s} \right)\left(M_2^{\frac{1}{2}} N_4^{\frac{1}{2}}\right)\left(N_6^{-\frac{3}{2}} N_6^{-\frac{3}{2}} \right) \Vert u \Vert_{X^s}^4 \Vert u_{N_6} \Vert_{X^s}^2\\
\lesssim& a\left(M_1\right)M_2^{-\frac{1}{2}-s}N_4^{\frac{1}{2}-s}M_3^{-2s-3}N_6^{-6s-5} \Vert u \Vert_{X^s}^4 \Vert u_{N_6} \Vert_{X^s}\Vert u_{N_6} \Vert_{X^s}.
\end{align*} 
Hence the summation with respect to $N,N_1,\dots,N_6$ is handled if $s\geq - \frac{5}{6}$.

\textbf{Subcase 2.b.ii} $N_5\leq N_6 \approx M_3$. Then we use the bilinear estimates for $u_{N_4}u_{M_3}$ and $L^\infty$ bound for $u_{M_2}$ and the $L^6$ Strichartz estimate for the $u_{M_3},u_{N_5},u_{N_6}$. Then we have
\begin{align*} 
\vert \Lambda_2 \vert \lesssim& \frac{a\left(M_1\right)}{M_2M_3^{3}} M_3^{-4s-2}\left(M_2^{-s}M_3^{-2s} N_4^{-s}N_5^{-s}M_3^{-s} \right)   M_3^{-\frac{3}{2}}M_2^{\frac{1}{2}}  \left(M_3^{-\frac{1}{3}} N_5^{-\frac{1}{3}} M_3^{-\frac{1}{3}} \right)                    
\Vert u \Vert_{X^s}^4 \Vert u_{N_6} \Vert_{X^s}^2\\
\lesssim & a\left(M_1\right)M_2^{-\frac{1}{2}-s} M_3^{-7s-\frac{43}{6}}N_4^{-s}N_5^{-s-\frac{1}{3}} \Vert u_{M_3} \Vert_{X^s} \Vert u_{M_3} \Vert_{X^s} \Vert u\Vert_{X^s}^4. 
\end{align*} 
Hence the summation with respect to $N,N_1,\dots,N_6$ is handled if $s\geq - \frac{8}{9}$.

\textbf{Subcase 2.b.iii} $N_5\leq N_6 \ll M_3$.  In this case, we use the bilinear estimates for the products $u_{M_3}u_{N_5}$ and $u_{M_3}u_{N_6}$ and the $L^\infty$ bound for $u_{M_2},u_{N_4}$. Then we have 
\begin{align*}
\vert \Lambda_2 \vert \lesssim& \frac{a\left(M_1\right)}{M_2M_3^{3}} M_3^{-4s-2}\left(M_2^{-s}M_3^{-2s} N_4^{-s}N_5^{-s}N_6^{-s} \right)     \left(M_2^{\frac{1}{2}} N_4^{\frac{1}{2}}\right)\left(M_3^{-\frac{3}{2}} M_3^{-\frac{3}{2}} \right) \Vert u \Vert_{X^s}^4 \Vert u_{M_3} \Vert_{X^s}^2 \\
\lesssim & a\left(M_1\right)M_2^{-s-\frac{1}{2}}N_4^{-s+\frac{1}{2}}N_5^{-s}N_6^{-s}M_3^{-6s-8} \Vert u_{M_3} \Vert_{X^s}^2 \Vert u\Vert_{X^s}^4.
\end{align*}
Hence the summation with respect to $N,N_1,\dots,N_6$ is handled if $s\geq - \frac{17}{18}$. 
\end{proof} 

To finish the proof of energy estimate, we need to choose suitable symbol $a\left(\xi\right)$ in the previous sections. As in \cite{KT2007}, we need the following sequence:
\begin{align*}
\beta_N^0=&\frac{N^{2s}\Vert u_{0,N} \Vert_{L^2}^2 }{\Vert u_0 \Vert_{H^s}^2 }\\
\beta_{N}=&\sum_{M}2^{-\frac{\epsilon}{2}\vert \log N -\log M \vert  }\beta_{M}^0. 
\end{align*} 
These $\beta_N$ satisfy the following property

(i)  $N^{2s} \Vert u_{0,N} \Vert_{L^2}^2 \lesssim \beta_N \Vert u_0 \Vert_{H^s}^2$,

(ii)  $\sum \beta_N \lesssim 1$, 
 
(iii) $\beta_N$ is slowly varying in the sense that
\begin{align*}
\vert \log_2 \beta_{N}-\log_2\beta_M \vert \lesssim \frac{\epsilon}{2} \vert \log_2N-\log_2 M \vert.
\end{align*}
We want to show that
\begin{align}\label{eqn:closed.}
\sup\limits_t N_0^s \Vert u_{N_0}\left(t\right) \Vert_{L^2} \lesssim \beta_{N_0}^{\frac{1}{2}}\left( \Vert u_0 \Vert_{H^s}+\Vert u \Vert_{X^s}^3 \right).
\end{align}
Then by using the property (ii) we can conclude Proposition \ref{prop:energy bound.}

To prove \eqref{eqn:closed.} for some frequency $N_0$ we choose 
\begin{align*}
a_N=N^{2s}\max\left\{1,\beta_{N_0}^{-1} 2^{-\epsilon\vert \log_2N-\log_2N_{0} \vert }  \right\} 
\end{align*}
Correspondingly we take a function $a\left(\xi\right) \in S_{\epsilon}^s$ so that 
\begin{align*}
a\left(\xi\right)\approx a_N, \quad \vert \xi \vert \approx N.
\end{align*}
Then from the slowly varying condition, we obtain
\begin{align*}
\sum\limits_N a_N \Vert u_{0,N} \Vert_{L_x^2}^2 \lesssim \sum\limits_N N^{2s} \Vert u_{0,N} \Vert_{L_x^2}^2+2^{-\epsilon \vert \log_2N-\log_2N_0 \vert  }N^{2s}\beta_{N_0}^{-1}\Vert u_{0,N} \Vert_{L_x^2}^2\lesssim \Vert u_0 \Vert_{H^s}^2.
\end{align*}   
From $\Vert u \Vert_{\ell_N^2 L_t^\infty H^s}\ll 1$, we have $\sup\limits_t E_0\left(u(t)\right)\ll 1$. Recall that
\begin{align*}
\frac{d}{dt}\left(E_0+E_1 \right)\left(u\right)= \Lambda_6\left(u(t)\right).
\end{align*}
From Proposition \ref{prop:E_1 energy} the contribution of $E_1$ to the energy is controlled by $E_0$. Also, we use the energy estimate in Proposition \ref{proposition: energy estimate 6 linear} for this choice of $a$. Therefore, we obtain
\begin{align*}
\left( \sum\limits_N a\left(N\right) \Vert u_{N}\left(t\right) \Vert_{L_x^2}^2 \right)^{\frac{1}{2}} \lesssim \Vert u_0 \Vert_{H^s}+\Vert u \Vert_{X^s}^3. 
\end{align*}
At fixed frequency $N=N_0$, we obtain $\eqref{eqn:closed.}$.  

\section{Proof of Theorem \ref{thm:main theorem}}\label{sec:main theorem} 
In this section we prove our main Theorem \ref{thm:main theorem}. The remaining part is just to do standard bootstrapping argument with trilinear estimate $\eqref{eqn:trilinear estimate}$ and energy estimate $\eqref{eqn:energy estimates N.}$. Our method is similar to the argument in Koch-Tataru \cite{KT2007}. Before we prove Theorem \ref{thm:main theorem}, we collect ingredients we need:
\begin{align*}
\text{Linear:}&\quad\quad &\Vert u \Vert_{X^s}\lesssim& \Vert u \Vert_{\ell^2L_t^\infty H_x^s}+\Vert \vert u \vert^2 u \Vert_{Y^s}\quad &\left(\ref{lem:nonlinear estimate}\right),\\
\text{Nonlinear:}&\quad\quad &\Vert \vert u \vert^2 u \Vert_{Y^s}\lesssim& \Vert u \Vert_{X^s}^3 \quad &\left(\ref{prop:trilinear estimate}\right),\\
\text{Energy:}&\quad\quad &\Vert u \Vert_{\ell^2L_t^\infty H_x^s}\lesssim& \Vert u_0\Vert_{H^s}+\Vert u \Vert_{X^s}^3\quad  &\left(\ref{prop:energy bound.}\right).
\end{align*} 
As we mentioned in Remark \ref{rem:small data remark}, by rescaling the problem we consider small initial data.
Let $\epsilon>0$ be a small constant and suppose $\Vert u_0 \Vert_{H^s\left(\mathbb{R}\right)}<\epsilon$. Take a small $\delta$ so that $\epsilon \ll \delta \ll 1$. We denote by $A$ the set
\begin{align*}
A=\left\{ T\in [0,1]; \Vert u\Vert_{\ell_N^2L_t^\infty H^s\left([0,T]\times\mathbb{R}\right)}\leq 2\delta, \; \Vert u \Vert_{X^s\left([0,T]\times \mathbb{R} \right)}\leq 2\delta      \right\}.
\end{align*}
We want to show that $A=\left[0,1\right]$. Clearly $A$ is not empty and $0\in A$. We need to prove that it is closed and open.
From the definition, the norms used in $A$ are continuous with respect to $T$ and hence $A$ is closed. 

Let $T \in A$. By using Proposition \ref{prop:energy bound.}, we have
\begin{align*}
\Vert u \Vert_{\ell_N^2L_t^\infty H^s\left([0,T]\times \mathbb{R} \right) }\lesssim \epsilon +\delta^3,
\end{align*}
and by Proposition \ref{lem:nonlinear estimate} and \ref{prop:trilinear estimate}, we have
\begin{align*}
\Vert u \Vert_{X^s\left([0,T]\times \mathbb{R} \right)}\lesssim \epsilon +\delta^3.
\end{align*}
Hence by taking $\epsilon$ and $\delta$ sufficiently small, we conclude that
\begin{align*}
\Vert u\Vert_{\ell_N^2 L_t^\infty H^s\left([0,T]\times \mathbb{R} \right)}\leq \delta, \quad \Vert u\Vert_{X^s\left([0,T]\times \mathbb{R} \right)}\leq \delta.
\end{align*}
Since the norms are continuous with respect to $T$, it follows that a neighborhood of $T$ is in $A$. Therefore $A=\left[0,1\right]$ and hence we prove Theorem \ref{thm:main theorem} .

\bibliographystyle{amsplain}
\bibliography{ddocument}

\end{document}